\documentclass[a4paper,12pt]{amsart}
\usepackage{amssymb,amstext,amsmath,amscd,amsthm,amsfonts,enumerate}

\usepackage{multirow}
\usepackage{booktabs}
\usepackage{array}
\usepackage{xcolor}
\usepackage{graphicx}
\usepackage{colortbl}
\usepackage{enumitem}
\usepackage{booktabs}
\usepackage{tikz}
\usetikzlibrary{arrows.meta}
\usepackage{subcaption}

\usepackage[margin=2.5cm]{geometry}
\theoremstyle{plain}
\newtheorem{thm}{Theorem}[section]
\newtheorem{theorem}[thm]{Theorem}

\newtheorem{proposition}[thm]{Proposition}
\newtheorem{lemma}[thm]{Lemma}

\newtheorem{corollary}[thm]{Corollary}

\newtheorem*{claim*}{Claim}
\newtheorem*{maintheorem*}{Main Theorem}
\newtheorem*{corollary*}{Corollary}

\theoremstyle{definition}

\newtheorem{example}[thm]{Example}

\newtheorem{remark}[thm]{Remark}

\newtheorem{ques}[thm]{Question}
\newtheorem{observation}[thm]{Observation}

\numberwithin{equation}{thm}

\definecolor{mylightgray}{RGB}{220,220,220}

\newcommand{\calR}{\mathcal{R}}
\newcommand{\calX}{\mathcal{X}}

\newcommand{\fkm}{\mathfrak{m}}

\def\height{\mathrm{ht}}
\def\Spec{\operatorname{Spec}}

\def\gr{\mbox{\rm gr}}

\def\red{\operatorname{red}}
\def\indeg{\operatorname{indeg}}

\def\V{\operatorname{V}}

\newcommand{\ceil}[1]{\lceil #1 \rceil}
\newcommand{\ord}{\operatorname{ord}}

\title{Normality of monomial ideals in three variables}
\author{Maki Ataka}
\address{Department of Mathematics, School of Science and Technology, Meiji University, 1-1-1 Higashi-mita, Tama-ku, Kawasaki 214-8571, Japan}
\email{atk.maki142@gmail.com}
\author{Naoyuki Matsuoka}
\address{Department of Mathematics, School of Science and Technology, Meiji University, 1-1-1 Higashi-mita, Tama-ku, Kawasaki 214-8571, Japan}
\email{naomatsu@meiji.ac.jp}

\thanks{The second author was partially supported by JSPS Grant-in-Aid for Scientific Research (C) 25K06941.}
\thanks{Statement on the use of AI. The authors used OpenAI ChatGPT Work (accessed in August 2026) to assist in checking the completeness of the case distinctions and the internal consistency of certain weight computations, tables, and examples. The tool was not used in formulating the main results or in conducting the original research. All suggested corrections were independently verified by the authors, who take full responsibility for the entire content of the manuscript.}

\begin{document}
\maketitle

\tableofcontents

\begin{abstract}
An ideal $I$ in a Noetherian ring is called \textit{normal} if $I^n$ is integrally closed for all $n \geq 1$. 
Zariski proved that in two-dimensional regular local rings, every integrally closed ideal is normal. 
However, in dimension three and higher, this is no longer true in general, 
including monomial ideals in polynomial rings.

In this paper, we study the normality of integrally closed monomial ideals
in the polynomial ring $k[x,y,z]$ over a field $k$.
We prove that every such ideal with at most seven minimal monomial generators is normal,
thereby giving a sharp bound for normality in this setting.
The proof is based on a detailed case-by-case analysis,
combined with valuation-theoretic and combinatorial methods
via Newton polyhedra.
\end{abstract}

\section{Introduction}

Let $I$ be an ideal of a Noetherian ring $R$.
We say that $x \in R$ is \emph{integral} over $I$ if $x$ satisfies the following equation of the form
$$
x^\ell + a_1 x^{\ell-1} + \cdots + a_m = 0
$$
where $a_i \in I^i$. We define
$$
\overline{I} = \{x \in R \mid \text{$x$ is integral over $I$}\}
$$
and call it the \emph{integral closure} of $I$. We say that $I$ is \emph{integrally closed}, if the equality $I=\overline{I}$ holds, and we say that $I$ is \emph{normal}, if every power of $I$ is integrally closed. 
Throughout this paper, $\mu_R(I)$ and $\height_R I$ denotes the minimal number of generators and the height of $I$, respectively.

The study of normal ideals has a long history, beginning with Zariski's foundational work in the polynomial ring in two variables \cite{Z},
and later extended to regular local rings of dimension two \cite[Appendix~5]{ZS}.
Zariski \cite{Z} proved that in such rings, 
the product of integrally closed ideals is integrally closed (see also \cite[Chapter 14]{SH}).
Consequently, every integrally closed ideal is normal. 
Moreover, Lipman proved that the above normality result remains valid for two-dimensional rational singularities \cite{Lipman}.

In dimension $d \geq 3$, however, even for polynomial rings or regular local rings, the situation becomes more difficult, and a natural question arises:

\begin{ques}
	Does an integrally closed ideal remain normal in higher dimensions?
\end{ques}

In general, the answer is negative, even for the three dimensional case.
Such an example is obtained from the ideal
$I=\overline{(x^7,y^3,z^2)}$ in the polynomial ring $A=k[x,y,z]$ (see \cite[Exercise 1.14]{SH}). 
A detailed analysis of this example is given in Example~\ref{ex:counterexample}.

This demonstrates that the situation in higher dimensions is fundamentally different from 
the two-dimensional case, even when the base ring remains regular.
Hence additional structure on $I$ is necessary to recover normality.
One natural approach is to impose numerical constraints on the ideal,
such as bounding the number of generators.

In this direction, several positive results are known.
Let $(R,\fkm)$ be a $d$-dimensional regular local ring and let $I$ be an $\fkm$-primary ideal of $R$.
In this setting, it is known that if $I$ is integrally closed, then $I$ is normal in the following cases:

\begin{enumerate}
\item When $\mu_R(I) = d$, namely, when $I$ is a parameter ideal, by Goto \cite{G}.
\item When $\mu_R(I) = d+1$, by Ciuperc\u{a} \cite{C}.
\item When $\mu_R(I) = d+2$, by Endo, Goto, Hong, and Ulrich \cite{EndoHongUlrich}, whose result applies to regular rings, not necessarily local.
\end{enumerate}

In a different direction, Reid, Roberts, and Vitulli \cite{RRV}
studied normality for monomial ideals.
They showed that for monomial ideals in $d$ variables,
if $I, I^2, \ldots, I^{d-1}$ are all integrally closed, then $I$ is normal.
In particular, when $d=3$, it suffices to check the integral closedness of $I$ and $I^2$.

From now on, we focus on monomial ideals in polynomial rings $A$ over a field $k$.

Indeed, every integrally closed monomial ideal in two variables 
is normal, thanks to Zariski's theorem mentioned above.

Similarly, for every zero-dimensional monomial ideal $I$ in $d$ variables, it follows from (1)--(3), combined with localization,
that $I$ is normal, if $I = \overline{I}$ and $\mu_A(I) \le d+2$.
In particular, when $d=3$, this covers all cases with $\mu_A(I)\le 5$.
We provide a direct proof of this fact in Section~\ref{sec:main_results}.

Beyond this range, Endo, Goto, Hong, and Ulrich \cite{EndoHongUlrich}
proved the following result for homogeneous ideals, not restricted to the monomial case.

\begin{theorem}[{\cite[Theorem~3.4]{EndoHongUlrich}}]\label{EGHU}
Let $A = k[x_1, x_2, \ldots, x_d]$ be a polynomial ring over a field $k$ of characteristic zero.
Let $I$ be a zero-dimensional ideal of $A$ generated by $d+3$ homogeneous polynomials.
If $I$ is integrally closed, then $I$ is normal. 
\end{theorem}

Note that Theorem \ref{EGHU} covers the case $\mu_A(I)=d+3$ for homogeneous ideals over fields of characteristic zero.
Indeed, their proof relies on methods that are inherently characteristic-dependent.
In the monomial setting, however, integral closedness can be characterized purely in terms of the Newton polyhedron, and hence is independent of the base field.
Consequently, their result implies the normality of monomial ideals $I$ with $\mu_A(I)=d+3$ and $I=\overline{I}$ over an arbitrary field.

Since our main interest lies in the next case beyond this range,
namely $\mu_A(I)=7$ in three variables,
we also include a self-contained proof of the case $\mu_A(I)\le 6$.
Although this case is already covered by the above results,
the argument presented here illustrates the method used throughout the paper
and prepares the reader for the more involved analysis in the seven-generator case.

With this perspective in mind, we now state our main result,
which extends the previously known bound in the monomial setting.

\begin{maintheorem*}\label{maintheorem_intro}
	Let $k$ be a field and $A=k[x,y,z]$ be the polynomial ring in three variables.
	Let $I$ be a monomial ideal with $\height_AI = 3$ and $I = \overline{I}$. 
	If $\mu_A(I) \le 7$, then $I$ is normal.
\end{maintheorem*}

Our proof method also yields the following extension to higher dimensions.

\begin{corollary*}\label{cor:higher_dim_intro}
	Let $k$ be a field, $d \geq 3$, and $A = k[x_1,\ldots,x_d]$ the polynomial ring.
	Let $I$ be a monomial ideal with $\height_AI=d$ and $I=\overline{I}$.
	If $\mu_A(I) \leq d+4$, then $I$ is normal.
\end{corollary*}

\begin{remark}
The bound $\mu_A(I)\le 7$ in Main Theorem is sharp.
The counterexample $I=\overline{(x^7, y^3, z^2)}$ mentioned above 
has eight generators and fails to be normal. 
This arises from a natural configuration of exponent vectors
that cannot occur in the classification for $\mu_A(I)\le 7$,
reflecting a genuine structural change in the geometry of the Newton polyhedron. 
More details are given in Example~\ref{ex:counterexample}.
Similar counterexamples in higher dimensions are obtained by adjoining variables.
\end{remark}

In what follows, throughout this paper, by an \emph{integrally closed monomial ideal}
we mean a monomial ideal with $I=\overline{I}$.

The proof of Main Theorem is based on a detailed case-by-case analysis.
After classifying all zero-dimensional integrally closed monomial ideals with six or seven generators
via their restrictions to polynomial rings in two variables, 
we verify normality in each case through valuation-theoretic tests combined with Reid-Roberts-Vitulli's criterion.

\medskip

We now discuss the geometric interpretation of our results via Rees algebras.
The Rees algebra $\calR(I)$ of an ideal $I$ has a natural geometric interpretation
as the coordinate ring of the blowup of $\Spec(A)$ along the closed subscheme $\V(I)$ defined by $I$.
In a normal domain, the normality of $\calR(I)$ is equivalent to the normality of the ideal $I$.

When $I$ is a monomial ideal in a polynomial ring, the Rees algebra $\calR(I)$ is an affine semigroup ring.
Therefore, if, in addition, $I$ is normal, Hochster's theorem \cite{Hochster} implies that
$\calR(I)$ is a normal Cohen-Macaulay ring with favorable geometric properties.

This viewpoint motivates the study of normal monomial ideals via affine semigroup rings.

\medskip

The organization of this paper is as follows.
Section~\ref{sec:preliminaries} collects basic facts on monomial ideals
and integral closures that will be used throughout the paper.
In Section~\ref{sec:main_results}, we prove the main theorem for
integrally closed monomial ideals $I$ in three variables with $\height_A I = 3$ and 
$\mu_A(I)\le 7$, reducing to the two-variable case when possible
and performing a detailed analysis in the remaining cases.
Section~\ref{sec:consequences} discusses further consequences of the main result,
including an extension to higher-dimensional polynomial rings
and structural results for integrally closed monomial ideals. Furthermore, we explore Rees algebras
and reduction numbers of normal monomial ideals.

\medskip

Throughout this paper, let $k$ be a field, and $\mathbb{N}$ denotes the set of positive integers.
In addition, for a real number $\alpha$, $\ceil{\alpha}$ denotes the minimum integer $n$ satisfying $\alpha \le n$. Note that $a \le b\ceil{\frac{a}{b}} \le a+b-1$ for any $a,b \in \mathbb{N}$.

\section{Preliminaries}\label{sec:preliminaries}

We collect several well-known results on integrally closed ideals that will be used throughout this paper.

\begin{lemma}[{cf. \cite[Proposition 1.5.2]{SH}}]\label{dvr}
Let $R$ be a discrete valuation ring. Then every ideal of $R$ is integrally closed.
\end{lemma}

\begin{theorem}[\cite{Z}, {\cite[Appendix 5, Theorem 2']{ZS}}]\label{zariski}
	Let $R$ be a $2$-dimensional regular local ring. Then the product of integrally closed ideals of $R$ is integrally closed. Thus, every integrally closed ideal of $R$ is normal.
\end{theorem}

\begin{theorem}[{\cite{RRV}, \cite[Theorem 1.4.10]{SH}}] \label{thm:RRV}
Let $I$ be a monomial ideal in a polynomial ring in $d$ variables over a field. If $I^i$ is integrally closed for all $1 \leq i \leq d-1$, then $I$ is normal.
\end{theorem}

In general, the integral closure of a monomial ideal $I$ in a polynomial ring
can be described in terms of the convex hull of the Newton polyhedron
associated with the generators of $I$.

Let $k$ be a field and let $A = k[x,y]$ be the polynomial ring in two variables.
For a monomial ideal $I \subseteq A$, we define
$$
\indeg(I) = \min \{ a+b \mid x^a y^b \in I \}.
$$

\begin{lemma}\label{lem:order_mingens}
Let $I$ be an integrally closed monomial ideal in $A = k[x,y]$ with
$\height_AI=2$. Then $\mu_A(I) = \indeg(I) + 1$.
\end{lemma}

\begin{proof}
This follows immediately from the description of integrally closed monomial
ideals in two variables via Newton polyhedra.

Alternatively, we may localize $A$ at $\fkm = (x,y)$ and apply the general
theory of integrally closed ideals in two-dimensional regular local rings
(e.g.\ \cite[Chapter~14]{SH}).
\end{proof}

The classification of integrally closed monomial ideals in two variables
can be obtained directly from the Newton polyhedron.
In particular, integral closures of monomial ideals are also monomial ideals.
We record below the explicit forms needed for later arguments.
Although the proofs are not exhaustive, we include enough details to clarify how the numerical conditions naturally arise.

\begin{proposition}\label{prop:2-variables}
Let $I$ be a monomial ideal in $A=k[x,y]$ with $\height_AI = 2$. 
Then we have the following.

\begin{enumerate}
	\item When $\mu_A(I) = 3$, $I$ is integrally closed if and only if, after switching $x$ and $y$ if necessary, we have one of the following.
	\begin{enumerate}
		\item $I = (x^a, y^b, xy)$ where $a, b \geq 3$.
		\item $I = (x^2, y^b, xy^{b_1})$ where $b \geq 2$ and $1 \leq b_1 \leq \lceil \frac{b}{2} \rceil$.
	\end{enumerate}

	\item When $\mu_A(I) = 4$, $I$ is integrally closed if and only if, after switching $x$ and $y$ if necessary, we have one of the following.
	\begin{enumerate}
		\item $I = (x^a, y^b, x^2y, xy^{b_2})$ where $a, b \geq 4$ and $2 \leq b_2 \leq \lceil \frac{b+1}{2} \rceil$.
		\item $I = (x^3, y^b, x^2y^{b_1}, xy^{b_2})$ where $b \geq 3$, $1 \leq b_1 < b_2 < b$, $1 \leq b_1 \leq \lceil \frac{b}{3} \rceil$, and $2b_1 - 1 \leq b_2 \leq \lceil \frac{b + b_1}{2} \rceil$.
		\end{enumerate}
	\end{enumerate}
\end{proposition}

\begin{proof}
We give a proof of the necessity of the numerical conditions for integral closedness.
Throughout the proof, we repeatedly use the following observation:
if there exists a monomial $m \notin I$ such that $m^2 \in I^2$,
then $m \in \overline{I} \setminus I$, contradicting the assumption $I=\overline{I}$.

\medskip

(1) Suppose that $I=(x^a, y^b, x^{a_1}y^{b_1})$ is integrally closed, where
$1 \le a_1 < a$ and $1 \le b_1 < b$.
In particular, we always have $a,b \ge 2$.

If $a,b \ge 3$, then Lemma~\ref{lem:order_mingens} implies $a_1=b_1=1$.
Hence we may assume $\min\{a,b\} = 2$. Then, without loss of generality, we also may assume $a=2$, so $a_1=1$.
If $b_1 > \ceil{\frac{b}{2}}$, then $xy^{b_1-1} \notin I$ while
$(xy^{b_1-1})^2 = x^2 y^{2b_1-2} \in I^2$, yielding a contradiction.
Therefore, the inequality
$
b_1 \le \ceil{\frac{b}{2}}
$
is necessary in this case.

\medskip

(2) Suppose that
$I=(x^a, x^{a_1}y^{b_1}, x^{a_2}y^{b_2}, y^b)$ is integrally closed,
where $1 \le a_2 < a_1 < a$ and $1 \le b_1 < b_2 < b$.

If $a,b \ge 4$, then Lemma~\ref{lem:order_mingens} implies
$a_1+b_1=3$ or $a_2+b_2=3$.
Without loss of generality, we may assume $a_1=2$ and $b_1=1$, hence $a_2=1$.
If $b_2 > \ceil{\frac{b+1}{2}}$, then $xy^{b_2-1} \notin I$ while
$(xy^{b_2-1})^2 \in I^2$, yielding a contradiction.
Thus we obtain
$b_2 \le \ceil{\frac{b+1}{2}}$.

Next, suppose $\min\{a,b\}=3$.
We may assume $a=3$, so $a_1=2$ and $a_2=1$.
If $b_1 > \ceil{\frac{b}{3}}$, then $x^2y^{b_1-1} \in \overline{I}\setminus I$,
yielding a contradiction.
Hence $b_1 \le \ceil{\frac{b}{3}}$

If $b_2 \le 2b_1-2$, then $\ceil{\frac{b_2}{2}} \le b_1-1$,
and $x^2 y^{\ceil{\frac{b_2}{2}}} \in \overline{I}\setminus I$,
again a contradiction.
Therefore we must have $b_2 \ge 2b_1-1$.

Finally, we get the bound $b_2 \le \ceil{\frac{b+b_1}{2}}$ by considering the monomial $xy^{b_2-1}$.

This completes the proof of the necessity of the conditions in (1) and (2).

\medskip

The sufficiency follows from the fact that, under these conditions,
the Newton polyhedron of $I$ contains no interior lattice points beyond the generators.
This can be verified by a direct inspection of the corresponding Newton polyhedra.
For illustration, see Figures~\ref{fig:three-configs}, which depict typical configurations for the case (2).
In these figures, the shaded region represents the set
$\overline{I} \setminus I$ in the exponent lattice.
Thus, $I$ is integrally closed if and only if this region
contains no lattice points.
These figures are meant only to indicate the geometric mechanism.
A complete verification of all cases would be routine and is omitted.
\end{proof}

\begin{figure}[h]
\centering
\begin{subfigure}[h]{0.23\textwidth}
\centering
\begin{tikzpicture}[scale=0.48]
  \def\a{5.0}   
  \def\b{8.0}   

  \draw[thick] (0,\b) -- (2,1);
  \draw[thick] (2, 1) -- (\a,0);

  \draw[->] (-0.3,0) -- (\a+0.6,0);
  \draw[->] (0,-0.3) -- (0,\b+0.6);

  \draw[dashed] (2,0) -- (2,1);
  \draw[dashed] (0,1) -- (2,1);

  \node[left]  at (0,\b) {$b$};
  \node[left]  at (0,1) {1};
  \node[below] at (2,0) {$2$};
  \node[below] at (\a,0) {$a$};

  \draw[very thick, gray!25!black]
    (0,\b) -- (1,\b)
    -- (1,5)
    -- (2,5)
    -- (2,1)
    -- (\a,1)
    -- (\a,0);

\fill[gray!25]
  (0,\b)
  -- (1,\b)
  -- (1,5)
  -- (2,5)
  -- (2,1)
  -- (\a,1)
  -- (\a,0)
  -- (2,1)
  -- (0,\b)
  -- cycle;
\end{tikzpicture}
\caption{\tiny $a,b\ge4$, $b_2 = \ceil{\frac{b+1}{2}}$}
\label{fig:case-i}
\end{subfigure}
\hfill
\begin{subfigure}[h]{0.23\textwidth}
\centering
\begin{tikzpicture}[scale=0.48]
  \def\a{5.0}   
  \def\b{8.0}   

  \draw[thick] (0,\b) -- (1,3);
  \draw[thick] (1,3) -- (2,1);
  \draw[thick] (2, 1) -- (\a,0);

  \draw[->] (-0.3,0) -- (\a+0.6,0);
  \draw[->] (0,-0.3) -- (0,\b+0.6);

  \draw[dashed] (2,0) -- (2,1);
  \draw[dashed] (0,1) -- (2,1);

  \node[left]  at (0,\b) {$b$};
  \node[left]  at (0,1) {1};
  \node[below] at (2,0) {$2$};
  \node[below] at (\a,0) {$a$};

  \draw[very thick, gray!25!black]
    (0,\b) -- (1,\b)
    -- (1,3)
    -- (2,3)
    -- (2,1)
    -- (\a,1)
    -- (\a,0);

\fill[gray!25]
  (0,\b)
  -- (1,\b)
  -- (1,3)
  -- (2,3)
  -- (2,1)
  -- (\a,1)
  -- (\a,0)
  -- (2,1)
  -- (1,3)
  -- (0,\b)
  -- cycle;
\end{tikzpicture}
\caption{\tiny $a,b\ge4$, $b_2 < \ceil{\frac{b+1}{2}}$}
\label{fig:case-i_2}
\end{subfigure}
\begin{subfigure}[h]{0.23\textwidth}
\centering
\begin{tikzpicture}[scale=0.48]
  \def\a{3.0}   
  \def\b{7.0}   

  \draw[thick] (0,\b) -- (1,3);
  \draw[thick] (1, 3) -- (\a,0);

  \draw[->] (-0.3,0) -- (\a+0.6,0);
  \draw[->] (0,-0.3) -- (0,\b+0.6);


  \node[left]  at (0,\b) {$b$};
  \node[below] at (\a,0) {$3$};

  \draw[very thick, gray!25!black]
    (0,\b) -- (1,\b)
    -- (1,3)
    -- (2,3)
    -- (2,2)
    -- (\a,2)
    -- (\a,0);

\fill[gray!25]
  (0,\b)
  -- (1,\b)
  -- (1,3)
  -- (2,3)
  -- (2,2)
  -- (\a,2)
  -- (\a,0)
  -- (1,3)
  -- (0,\b)
  -- cycle;
\end{tikzpicture}
\caption{\tiny $a=3$, $b_1 = \ceil{\frac{b}{3}}$, $b_2 =2b_1-1$}
\label{fig:case-ii}
\end{subfigure}
\hfill
\begin{subfigure}[h]{0.23\textwidth}
\centering
\begin{tikzpicture}[scale=0.48]
  \def\a{3.0}   
  \def\b{8.0}   

  \draw[->] (-0.3,0) -- (\a+0.6,0);
  \draw[->] (0,-0.3) -- (0,\b+0.6);


  \node[left]  at (0,\b) {$b$};
  \node[below] at (\a,0) {$3$};

  \draw[very thick, gray!25!black]
    (0,\b) -- (1,\b)
    -- (1,3)
    -- (2,3)
    -- (2,1)
    -- (\a,1)
    -- (\a,0);

\fill[gray!25]
    (0,\b) -- (1,\b)
    -- (1,3)
    -- (2,3)
    -- (2,1)
  -- (\a,1)
  -- (\a,0)
  -- (2,1)
  -- (1,3)
  -- (0,\b)
  -- cycle;

  \draw[thick] (0,\b) -- (1,3);
  \draw[thick] (1, 3) -- (2,1);
  \draw[thick] (2, 1) -- (\a,0);

\end{tikzpicture}
\caption{\tiny $a =3$, $b_1 < \ceil{\frac{b}{3}}$, $b_2 < \ceil{\frac{b+b_1}{2}}$}
\label{fig:case-iii}
\end{subfigure}

\caption{Typical Newton polyhedra in two variables with $\mu_A(I) = 4$}
\label{fig:three-configs}
\end{figure}
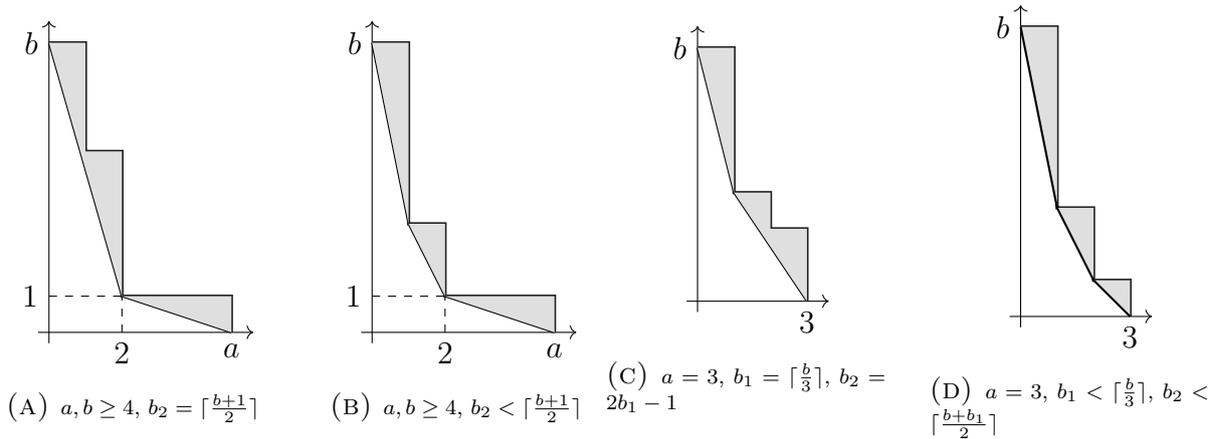

\begin{remark}
While the classification of integrally closed monomial ideals in two variables, as mentioned above, is relatively straightforward via Newton polyhedra, the situation changes drastically when moving to three variables. 
As demonstrated by Example~\ref{ex:counterexample}, nontrivial obstructions to normality can occur starting from dimension three. 

Although the classification of integrally closed monomial ideals with 6 or 7 generators can still be carried out explicitly (see Subsections \ref{sec:6gens} and \ref{sec:7gens}), verifying that $I^2$ is integrally closed becomes significantly more challenging. 
The main difficulty in our proof lies precisely in the exhaustive case-by-case verification required to show that $I^2$ is integrally closed for each configuration.
\end{remark}

\section{Main Results}\label{sec:main_results}

The purpose of this section is to prove the following theorem, which coincides with Main Theorem stated in Section~1.

\begin{thm}\label{maintheorem}
	Let $k$ be a field and $A=k[x,y,z]$ be the polynomial ring in three variables.
	Let $I$ be an integrally closed monomial ideal with $\height_AI = 3$. 
	If $\mu_A(I) \le 7$, then $I$ is normal.
\end{thm}

We prove Theorem~\ref{maintheorem} by considering the cases
$\mu_A(I)\le 5$, $\mu_A(I)=6$, and $\mu_A(I)=7$ separately,
which are treated in Subsections 3.1, 3.3, and 3.4, respectively.
The cases $\mu_A(I)\le 5$ and $\mu_A(I)=6$ are already known in the literature;
the latter follows from the result of Endo-Goto-Hong-Ulrich \cite{EndoHongUlrich} as discussed in the introduction.
Nevertheless, we include complete proofs for these cases, since they illustrate the method used throughout the paper and serve as a preparation for the more involved case $\mu_A(I)=7$.

\subsection{The Case $\mu_A(I) \le 5$}

\begin{lemma}\label{lem:key}
Let $A = k[x_1, \ldots, x_d]$ and let $I$ be an integrally closed monomial ideal with $\height_AI = d$. 
For $1 \leq i < j \leq d$ with $x_i, x_j \notin I$, there exists a minimal generator of $I$ of the form $x_i^\ell x_j^m$ with $\ell, m \ge 1$.
\end{lemma}

\begin{proof}
	Let $1 \le i < j \le d$ with $x_i, x_j \notin I$. 
	By renumbering the variables, we may assume $(i,j)= (1,2)$, so that $x_1, x_2 \notin I$. 
	We consider $J=I \cap k[x_1, x_2]$. Then $J$ is an integrally closed monomial ideal with order at least $2$.
	Therefore, thanks to Lemma \ref{lem:order_mingens}, $\mu_{k[x_1, x_2]}(J) = \indeg(J) + 1 \ge 3$. 
	Thus we can find a minimal generator $f \in k[x_1, x_2]$ of $J$ which is neither a power of $x_1$ nor a power of $x_2$. This means that $f = x_1^\ell x_2^m$ for some $\ell, m \ge 1$. Since $f$ is a minimal generator of $J$, $f$ is also a minimal generator of $I$.
\end{proof}

The following is a consequence of Lemma~\ref{lem:key}.
Although the assertion below admits a higher-dimensional generalization,
we confine ourselves here to the three-variable case,
since the higher-dimensional case will be discussed in Section~\ref{sec:higher_dim}.

\begin{corollary}[cf. the much more general results in \cite{ C, EndoHongUlrich, G}]\label{cor:5gens}
Let $A = k[x,y,z]$ and $I$ be an integrally closed monomial ideal with $\height_AI = 3$. 
If $\mu_A(I) \le 5$, then $I$ is normal.
\end{corollary}

\begin{proof}
	Suppose that $\mu_A(I) \le 5$. To prove the normality of $I$, we may assume $x,y,z \notin I$. 
	By Lemma~\ref{lem:key}, we can choose minimal generators $f_1 \in I \cap k[x, y]$, $f_2 \in I \cap k[x, z]$, and $f_3 \in I \cap k[y, z]$ of $I$.
	Together with the powers of $x, y, z$ appearing in a minimal generating set of $I$, this implies that $\mu_A(I) \ge 6$.
	Therefore, at least one of $x$, $y$, and $z$ is in $I$.
	Without loss of generality, we may assume $z \in I$.
	Then $I/(z)$ is an integrally closed monomial ideal in $A/(z) \cong k[x,y]$, and hence $I/(z)$ is normal.
	This readily implies that $I$ is normal.
\end{proof}

The proof above relies on the fact that, for monomial ideals,
the existence of generators of order one is automatically guaranteed by Lemma~\ref{lem:key}.
While reductions to lower-dimensional cases are present in \cite{C, EndoHongUlrich},
their validity in the general setting depends on delicate ring-theoretic arguments.

\subsection{Strategy for proving the main result in the cases $\mu_A(I)=6$ and $7$}

In this subsection, we explain the strategy for proving normality
in the cases $\mu_A(I) = 6$ and $\mu_A(I) = 7$.

Throughout this subsection, let $A=k[x,y,z]$, $\fkm = (x,y,z)$ and $I$ be a monomial ideal of $A$ with $\height_AI = 3$ (for the moment, not necessarily $6 \le \mu_A(I) \le 7$).
To prove $I = \overline{I}$, it suffices to show that $g \notin \overline{I}$ for any monomial $g \in (I:\fkm) \setminus I$.

In order to explain how to prove $g \notin \overline{I}$, we define the ring homomorphism $\varphi_\lambda: A \to k[[t]]$ such that $\varphi_\lambda(x) = t^p$, $\varphi_\lambda(y) = t^q$, $\varphi_\lambda(z) = t^r$ for $\lambda = (p,q,r) \in \mathbb{N}^3$.
Moreover, we define $\ord_\lambda(h) = \ord(\varphi_\lambda(h))$ for $h \in k[x,y,z]$ and $\ord_\lambda(I) = \min \{\ord_\lambda(h) \mid h \in I\}$,
where $\ord$ denotes the order in the formal power series ring $k[[t]]$.

\begin{proposition}\label{prop:ord_test}
Suppose that there exists $\lambda = (p,q,r) \in \mathbb{N}^3$ such that $\ord_\lambda (g) < \ord_\lambda(I)$. Then $g \notin \overline{I}$.
\end{proposition}

\begin{proof}
We put $V=k[[t]]$.
Assume that $g \in \overline{I}$.
By Lemma \ref{dvr}, every ideal of $V$ is integrally closed.
In particular,
$$\varphi_\lambda(g) \in \varphi_\lambda(\overline{I})V \subseteq \overline{\varphi_\lambda(I)V} = \varphi_\lambda(I)V,$$
which is equivalent to $\ord_\lambda(g) \ge \ord_\lambda(I)$, since $\varphi_\lambda(I)V = (t^{\ord_\lambda(I)})$.
This contradicts our assumption.
\end{proof}

Put $I_1 = I \cap k[x,y]$, $I_2 = I \cap k[x,z]$, and $I_3 = I \cap k[y,z]$.
It is immediate that if $I$ is integrally closed, then so is $I_i$ for each $i=1,2,3$.
Conversely, under an additional assumption, the following holds.

\begin{lemma}\label{lem:xy}
Suppose that at least one of $xy, xz, yz$ is in $I$ and $I_i$ is integrally closed for every $i=1,2,3$. Then $I$ is integrally closed.
\end{lemma}

\begin{proof}
	We may assume $xy \in I$.
	Suppose that $I \subsetneq \overline{I}$ and take any monomial $g \in \overline{I} \setminus I$.
	Since $xy \in I$, $g$ has the form $x^\alpha z^\gamma$ or $y^\beta z^\gamma$.
	Without loss of generality, we may assume $g = x^\alpha z^\gamma$.
	Consider the projection $\rho : A \to k[x,z]$, namely, $\rho(y) = 0$. 
	Then $g=\rho(g)$ and, since $I$ is a monomial ideal, we have $\rho(I)=I\cap k[x,z]=I_2$.
	By persistence, $\rho(\overline{I})\subseteq \overline{\rho(I)}=\overline{I_2}$, hence $g\in \overline{I_2}$.
	Therefore, since $I_2$ is integrally closed, we have $g \in I_2 \subseteq I$, a contradiction.
\end{proof}

In the rest of this section, we prove the normality of $I$ separately for the cases $\mu_A(I)=6$ and $\mu_A(I)=7$,
but the strategy is the same for both cases. We describe the procedure here.

First, we classify all integrally closed monomial ideals in each case.
The details will be given in each subsection, but the key tools are the fact that
$I_1, I_2, I_3$ are integrally closed if so is $I$, and Lemma \ref{lem:xy}.

From now on, suppose that $I$ is integrally closed and $6 \le \mu_A(I) \le 7$.
In order to prove the normality of $I$, by Theorem \ref{thm:RRV} it suffices to prove that $I^2$ is integrally closed.

Let $g \in (I^2:\fkm) \setminus I^2$ be a monomial. We will show that $g \notin \overline{I^2}$.

First, suppose that $g \notin (xyz)$.
Then $g \in k[x,y]$ or $g \in k[x,z]$ or $g \in k[y,z]$.
We may assume $g \in k[x, y]$.
Note that $I_1 = I \cap k[x,y]$ is an integrally closed monomial ideal in $k[x,y]$.
By Theorem \ref{zariski} $I_1$ is normal, in particular, $I_1^2$ is integrally closed.
By considering the projection $\rho:A \to k[x,y]$ as in the proof of Lemma \ref{lem:xy}, if we assume $g \in \overline{I^2}$, then we get $g \in \overline{I_1^2} = I_1^2 \subseteq I^2$ a contradiction.
Therefore, $g \notin \overline{I^2}$.

Thus, we may assume $g \in (xyz)$. We enumerate all such monomials $g$.
For each such $g$, we find $\lambda = (p,q,r) \in \mathbb{N}^3$ such that
$\ord_\lambda(g) < \ord_\lambda(I^2) = 2\ord_\lambda(I)$.
Then, by Proposition \ref{prop:ord_test}, we obtain $g \notin \overline{I^2}$, completing the proof.

\subsection{The Case $\mu_A(I) =6$} \label{sec:6gens}

\subsubsection{Classification of integrally closed monomial ideals}

Let $I$ be an integrally closed monomial ideal with $\mu_A(I)=6$ and $\height_AI = 3$.
Then there exist $a,b,c \in \mathbb{N}$ and monomials $f_1, f_2, f_3 \in A$ with 
$$I=(x^a, y^b, z^c, f_1, f_2, f_3).$$

If one of $a, b, c$ equals $1$ --- say, $c=1$ after relabeling variables if necessary --- then 
$I/(z)$ is an integrally closed monomial ideal in $A/(z) \cong k[x,y]$, 
and hence $I/(z)$ is normal by using Theorem \ref{zariski}, which implies $I$ is also normal.
Therefore, we may assume $a, b, c \ge 2$.

Set $A_1 = k[x,y]$, $A_2 = k[x,z]$, $A_3 = k[y,z]$, and $I_i = I \cap A_i$ for each $i=1,2,3$.
By Lemma \ref{lem:key}, we may assume $f_i \in A_i$ for each $i=1,2,3$.
We write $f_1 = x^{a_1} y^{b_1}$, $f_2 = x^{a_2}z^{c_2}$, and $f_3 = y^{b_3}z^{c_3}$ 
for some $a_1, a_2, b_1, b_3, c_2, c_3 \in \mathbb{N}$.
Furthermore, by using Proposition \ref{prop:2-variables}, we have the following classification.

\begin{thm}\label{thm:6gens}
	Let $a,b,c,a_1, a_2, b_1, b_3, c_2, c_3 \in \mathbb{N}$ with 
	$a, b, c \ge 2$, $1 \le a_1, a_2 < a$, $1 \le b_1, b_3 < b$, and $1 \le c_2, c_3 < c$.
	Put $$I=(x^a, y^b,z^c, x^{a_1} y^{b_1}, x^{a_2}z^{c_2}, y^{b_3}z^{c_3}).$$
	Then the following conditions are equivalent.
	
	\begin{enumerate}
		\item $I=\overline{I}$.
		\item $I_i = \overline{I_i}$ for every $i=1,2,3$.
		\item After permuting $x, y, z$ if necessary, one of the following is satisfied.
		\begin{enumerate}
		\item $a,b,c \ge 3$ and $a_1=a_2=b_1=b_3=c_2=c_3=1$.
		\item $a,b\ge 3$, $c=2$, $a_1=b_1=c_2=c_3=1$, $1 \le a_2 \le \ceil{\frac{a}{2}}$, 
		      and $1 \le b_3 \le \ceil{\frac{b}{2}}$.
		\item $b=c=2$, $b_1=b_3=c_2=c_3=1$, and $1 \le a_1, a_2 \le \ceil{\frac{a}{2}}$.
		\end{enumerate}
	\end{enumerate}
\end{thm}

\begin{proof}
Note that
$$
I_1 = (x^a, y^b, x^{a_1}y^{b_1})A_1,
\quad
I_2 = (x^a, z^c, x^{a_2}z^{c_2})A_2,
\quad
I_3 = (y^b, z^c, y^{b_3}z^{c_3})A_3,
$$
which are monomial ideals in two variables with three generators. The implication (1) $\Rightarrow$ (2) is trivial.

\textbf{(3) $\Rightarrow$ (1):} Suppose (3). Note that $xy \in I$ or $yz \in I$ in each case.
Moreover, by Proposition \ref{prop:2-variables}, $I_1, I_2, I_3$ are integrally closed.
Therefore, by Lemma \ref{lem:xy}, $I$ is integrally closed. Thus we have (1).

\textbf{(2) $\Rightarrow$ (3):} 
Suppose (2). 
By permuting $x,y,z$ if necessary, we may assume $a \ge b \ge c$.
If $c \ge 3$, we have $a_1 = a_2 = b_1 = b_3 = c_2= c_3= 1$ by Proposition \ref{prop:2-variables}.
Next, assume $c =2$ and $b \ge 3$. Since $c=2$, we have $c_2 = c_3=1$. Moreover, since $a \ge b \ge 3$, we have $a_1 = b_1 =1$ by Proposition \ref{prop:2-variables}. The remaining conditions also follow from Proposition \ref{prop:2-variables}.
Finally, assume $b=c=2$. Then $b_1 = b_3 = c_2 = c_3 =1$.
The ranges of $a_1$ and $a_2$ are given by Proposition \ref{prop:2-variables}.
\end{proof}

\subsubsection{Determining $(I^2:\fkm)\setminus I^2$}

Let $I$ be an integrally closed monomial ideal in $A=k[x,y,z]$ with $\height_AI = 3$ and $\mu_A(I) = 6$.
In this subsubsection, we explicitly determine the set
$$
\calX = \left\{g \in (I^2:\fkm)\setminus I^2 ~\middle|~ g = x^\alpha y^\beta z^\gamma~(\alpha, \beta, \gamma \in \mathbb{N}) \right\}.
$$
We use the notation in Theorem \ref{thm:6gens}. 


\bigskip

\noindent
\textbf{Case 1: $a_1=a_2=b_1=b_3=c_2=c_3=1$.}

In this case, $I=(x^a, y^b, z^c, xy,xz, yz)$.
It is straightforward to verify that $xyz \in (I^2:\fkm) \setminus I^2$, and thus $\calX = \{xyz\}$.

\medskip

\textit{Reduction Step.}
When $a,b,c \ge 3$, all cases are already covered by Case~1.
Hence, to treat the remaining cases, we may assume $c=2$ after permuting $a,b,c$ if necessary. Hence $c_2 = c_3 = 1$.

\bigskip

\noindent
\textbf{Case 2: $a,b \ge 3$, $c=2$, $c_2 = c_3 = 1$.}

In this case, we have $a_1 = b_1 = 1$ by Proposition \ref{prop:2-variables}.
Namely, $I=(x^a, y^b, z^2, xy, x^{a_2}z, y^{b_3}z)$.
If $a_2 = b_3 = 1$, this reduces to Case 1.
Therefore, we may assume $a_2 \ge 2$ or $b_3 \ge 2$.

We divide into two cases.

\medskip

\noindent
\underline{Subcase 2-a:} $a_2 = 1$.

In this case, we have $x^2yz, xyz^2 \in I^2$. 
Hence, if $g=x^\alpha y^\beta z^\gamma \in (xyz)$ satisfies $g \notin I^2$, then $\alpha = \gamma = 1$. Hence we need to consider $g = xy^\beta z$.

Since 
$$xy^{b_3}z \cdot x = x^2y^{b_3}z = (xy)^2y^{b_3-2}z \in I^2 \quad (\text{as } x^2y^2 \in I^2),$$
$$xy^{b_3}z \cdot y = xy^{b_3+1}z \in I^2, \text{ and }$$
$$xy^{b_3}z \cdot z = xy^{b_3}z^2 \in I^2 \quad (\text{as } xyz^2 \in I^2),$$
we have $xy^{b_3}z \in I^2:\fkm$.
Moreover, $xy^{b_3}z \notin I^2$ since it is not divisible by any minimal generator of $I^2$.
Thus, $xy^{b_3}z$ is the only monomial in $\calX$.

\medskip
\textit{Reduction Step.} 
Since the roles of $a$ and $b$ are symmetric, the case $b_3 = 1$ reduces to Subcase 2-a by symmetry.

\medskip

\noindent
\underline{Subcase 2-b:} $2 \le a_2 \le \ceil{\frac{a}{2}}$ and $2 \le b_3 \le \ceil{\frac{b}{2}}$.

Let $g=x^\alpha y^\beta z^\gamma \in (xyz)$ so that $g \notin I^2$.
Since $xyz^2 \in I^2$, we have $\gamma = 1$. Furthermore, $x^2y^2 \in I^2$ implies $\alpha = 1$ or $\beta =1$.
Therefore, we need to consider the two cases: $g = x^\alpha yz$ and $g = xy^\beta z$.
Then, by a computation similar to Subcase 2-a, we can conclude that $\calX = \{x^{a_2}yz, xy^{b_3}z\}$. 

This completes the analysis of Case 2.

\medskip

\textit{Reduction Step.}
Since the case of $a,b \ge 3$ is completed, it remains to consider the case when $a = 2$ or $b = 2$.
By swapping $a$ and $b$ if necessary, we may assume $b = 2$. Hence $b_1 = b_3 = 1$.
Moreover, if $a = 2$, this reduces to Case 1.
Therefore, we may assume $a \ge 3$.

\bigskip

\noindent
\textbf{Case 3: $a \ge 3$, $b=c=2$, $b_1 = b_3 = c_2 = c_3=1$.}

Since the roles of $b$ and $c$ are symmetric, we may assume $a_1 \ge a_2$ without loss of generality.
In this case, we have $I=(x^a, y^2, z^2, x^{a_1}y, x^{a_2}z, yz)$.

We divide into three cases.

\medskip

\noindent
\underline{Subcase 3-a:} $a_2 = 1$.

The case $a_1 = 1$ can be reduced to Case 1. So, we may assume $a_1 \ge 2$.
By a computation similar to Subcase 2-a, we get $\calX = \{x^{a_1}yz\}$.

\medskip

\textit{Reduction Step.}
It remains to consider the case $2 \le a_2 \le \ceil{\frac{a}{2}}$.
By a computation similar to Subcase 2-a, we see that $x^{a_2-1}y^2z, x^{a_2-1}yz^2 \in \calX$.
Since $y^3z, y^2z^2, yz^3 \in I^2$, other possible candidates for $\calX$ are limited to monomials
of the form $x^\alpha yz$ with $\alpha \ge 1$.

\medskip

\noindent
\underline{Subcase 3-b:} $a_1 + a_2 \le a$.

In this case, we have
$\calX = \{x^{a_2-1}y^2z, x^{a_2-1}yz^2, x^{a_1+a_2-1}yz\}.$

\medskip

\noindent
\underline{Subcase 3-c:} $a_1 + a_2 > a$.

In this case, we have
$\calX = \{x^{a_2-1}y^2z, x^{a_2-1}yz^2, x^{a-1}yz\}.$
Note that $a_1 + a_2 > a$ if and only if $a$ is odd and $a_1 = a_2 = \frac{a+1}{2}$, 
because $a_1, a_2 \le \ceil{\frac{a}{2}}$.

\medskip
 
\textit{Summary.}
We summarize the results of this subsubsection in Table \ref{tab:classification_6gen}.
Note that $b_1 = c_2 = c_3 = 1$ in all cases.
The symbol ``--'' indicates that no additional conditions are required beyond the prerequisites.

\begin{table}[h]
\centering
\caption{The candidate monomials $g$ for $\mu_A(I)=6$ case}
\label{tab:classification_6gen}
\small
\begin{tabular}{|c|c|c|c|c|c|c|l|l|}
\hline
\textbf{No.} & $a$ & $b$ & $c$ & $a_1$ & $a_2$ & $b_3$ & \textbf{\small Condition} & \textbf{Candidates $g$} \\
\hline
1 & -- & -- & -- & $1$ & $1$ & $1$ & -- & $xyz = xy^{b_3}z$ \\
\hline
2-a & $\ge 3$ & $\ge 3$ & $2$ & $1$ & $1$ & $\ge 2$ & -- & $xy^{b_3}z$ \\
\hline
2-b & $\ge 3$ & $\ge 3$ & $2$ & $1$ & $\ge 2$ & $\ge 2$ & -- & $x^{a_2}yz$, $xy^{b_3}z$ \\
\hline
3-a & $\ge 3$ & $2$ & $2$ & $\ge 2$ & $1$ & $1$ & -- & $x^{a_1}yz$ \\
\hline
3-b & $\ge 3$ & $2$ & $2$ & $\ge a_2$ & $\ge 2$ & $1$ & $a_1+a_2 \le a$ & $x^{a_2-1}y^2z$, $x^{a_2-1}yz^2$, $x^{a_1+a_2-1}yz$ \\
\hline
3-c & $\ge 3$ & $2$ & $2$ & $\ge a_2$ & $\ge 2$ & $1$ & $a_1+a_2 >a$ & $x^{a_2-1}y^2z$, $x^{a_2-1}yz^2$, $x^{a-1}yz$ \\
\hline
\end{tabular}
\end{table}

\subsubsection{Finding $\lambda =(p,q,r)$ with $\ord_\lambda(g) < \ord_\lambda(I^2)$}

In this subsubsection, for each candidate monomial $g$ listed in
Table~\ref{tab:classification_6gen},
we construct $\lambda = (p,q,r) \in \mathbb{N}^3$
such that
$$
\ord_\lambda(g) < \ord_\lambda(I^2)
$$
for every monomial ideal $I$ of the form considered in the corresponding cases.
By Proposition~\ref{prop:ord_test}, this inequality implies $g \notin \overline{I^2}$.

The choice of $\lambda$ is guided by the exponent structure of $g$ and the generators of $I$.
The verification is then performed by a finite exhaustive case analysis.

Throughout this subsubsection, we consider monomial ideals
$$
I = (x^a, y^b, z^c, x^{a_1}y, x^{a_2}z, y^{b_3}z)
$$
appearing in the classification given in the previous subsubsection.
Here, the exponents $a,b,c,a_1,a_2,b_3$ are arbitrary positive integers
satisfying the conditions of the corresponding case.
Note that, if $b = 2$, then $1 \le a_1, a_2 \le \ceil{\frac{a}{2}}$.

First, we recall the following basic properties.

\begin{itemize}
\item For a monomial $f=x^\alpha y^\beta z^\gamma$ corresponding to the vector $v=(\alpha, \beta, \gamma)$, we have $\ord_\lambda(f) = (\lambda, v)$, the inner product of $\lambda$ and $v$.
\item $\ord_\lambda(I) = \min \{ \ord_\lambda(f) \mid f \in \{x^a, y^b, z^c, x^{a_1}y, x^{a_2}z, y^{b_3}z\}\}$.
\item $\ord_\lambda(I^2) = 2\ord_\lambda(I)$.
\end{itemize}

\bigskip

\noindent
\underline{$g = xy^{b_3}z$}:

We divide into two cases: $2b_3 \le b$ and $2b_3 > b$.

\medskip

\noindent\textbf{When $2b_3 \le b$.}

We put $\lambda = (2b_3-1, 1, b_3)$. Then $\ord_\lambda(g) = 4b_3-1$, and we need $\ord_\lambda(I) \ge 2b_3$.

\begin{center}
\begin{tabular}{|c||c|c|c|c|c|c|}
\hline
$f$ & $x^a$ & $y^b$ & $z^c$ & $x^{a_1}y$ & $x^{a_2}z$ & $y^{b_3}z$ \\
\hline
$\ord_\lambda(f)$ & $a(2b_3-1)$ & $b$ & $cb_3$ & $a_1(2b_3-1)+1$ & $a_2(2b_3-1)+b_3$ & \cellcolor{mylightgray}$2b_3$ \\
\hline
\end{tabular}
\end{center}

Since $a\ge 2$, $a_1, a_2 \ge 1$, $b \ge 2b_3\ge 2$, and $c \ge 2$, we have $\ord_\lambda(I) = 2b_3$.

\medskip

\noindent\textbf{When $2b_3 > b$.}

Then $2b_3 = b+1$, since $b_3 \le \ceil{\frac{b}{2}}$.

We put $\lambda = (2b-2, 2, b)$. Then $\ord_\lambda(g) = 4b-1$, and we need $\ord_\lambda(I) \ge 2b$.

\begin{center}
\begin{tabular}{|c||c|c|c|c|c|c|}
\hline
$f$ & $x^a$ & $y^b$ & $z^c$ & $x^{a_1}y$ & $x^{a_2}z$ & $y^{b_3}z$ \\
\hline
$\ord_\lambda(f)$ & $a(2b-2)$ & \cellcolor{mylightgray}$2b$ & $cb$ & $2a_1(b-1)+2$ & $2a_2(b-1) + b$ & $2b + 1$ \\
\hline
\end{tabular}
\end{center}

Since $a,b,c \ge 2$ and $a_1, a_2 \ge 1$, we have $\ord_\lambda(I) = 2b$.

\medskip
\noindent\rule{\textwidth}{0.4pt}

\medskip

\noindent
\underline{$g = x^{a_1}yz$}:

This monomial appears only in Subcase~3-a, where
$I=(x^a,y^2,z^2,x^{a_1}y,xz,yz)$ and
$2\le a_1\le \ceil{\frac{a}{2}}$.
We divide into two cases: $2a_1\le a$ and $2a_1>a$.

\medskip

\noindent\textbf{When $2a_1\le a$.}

We put $\lambda=(1,a_1,2a_1-1)$. Then $\ord_\lambda(g)=4a_1-1$, and we need $\ord_\lambda(I)\ge 2a_1$.

\begin{center}
\begin{tabular}{|c||c|c|c|c|c|c|}
\hline
$f$ & $x^a$ & $y^2$ & $z^2$ & $x^{a_1}y$ & $xz$ & $yz$ \\
\hline
$\ord_\lambda(f)$ & $a$ & \cellcolor{mylightgray}$2a_1$ & $4a_1-2$ & \cellcolor{mylightgray}$2a_1$ & \cellcolor{mylightgray}$2a_1$ & $3a_1-1$ \\
\hline
\end{tabular}
\end{center}

Since $a\ge 2a_1$ and $a_1\ge 2$, we have $\ord_\lambda(I)=2a_1$.

\medskip

\noindent\textbf{When $2a_1>a$.}

Since $a_1\le \ceil{\frac{a}{2}}$, we have $2a_1=a+1$.
We put $\lambda=(2,a,2a-2)$. Then $\ord_\lambda(g)=4a-1$, and we need $\ord_\lambda(I)\ge 2a$.

\begin{center}
\begin{tabular}{|c||c|c|c|c|c|c|}
\hline
$f$ & $x^a$ & $y^2$ & $z^2$ & $x^{a_1}y$ & $xz$ & $yz$ \\
\hline
$\ord_\lambda(f)$ & \cellcolor{mylightgray}$2a$ & \cellcolor{mylightgray}$2a$ & $4a-4$ & $2a+1$ & \cellcolor{mylightgray}$2a$ & $3a-2$ \\
\hline
\end{tabular}
\end{center}

Since $a\ge 3$, we have $\ord_\lambda(I)=2a$.

\medskip
\noindent\rule{\textwidth}{0.4pt}

\medskip

\noindent
\underline{$g = x^{a_2}yz$}:

This monomial appears only in Subcase~2-b, where
$I=(x^a,y^b,z^2,xy,x^{a_2}z,y^{b_3}z)$,
$2\le a_2\le \ceil{\frac{a}{2}}$, and
$2\le b_3\le \ceil{\frac{b}{2}}$.
We divide into two cases: $2a_2\le a$ and $2a_2>a$.

\medskip

\noindent\textbf{When $2a_2\le a$.}

We put $\lambda=(1,2a_2-1,a_2)$. Then $\ord_\lambda(g)=4a_2-1$, and we need $\ord_\lambda(I)\ge 2a_2$.

\begin{center}
\begin{tabular}{|c||c|c|c|c|c|c|}
\hline
$f$ & $x^a$ & $y^b$ & $z^2$ & $xy$ & $x^{a_2}z$ & $y^{b_3}z$ \\
\hline
$\ord_\lambda(f)$ & $a$ & $b(2a_2-1)$ & \cellcolor{mylightgray}$2a_2$ & \cellcolor{mylightgray}$2a_2$ & \cellcolor{mylightgray}$2a_2$ & $b_3(2a_2-1)+a_2$ \\
\hline
\end{tabular}
\end{center}

Since $a\ge 2a_2$, $b\ge 3$, and $b_3\ge 2$, we have $\ord_\lambda(I)=2a_2$.

\medskip

\noindent\textbf{When $2a_2>a$.}

Since $a_2\le \ceil{\frac{a}{2}}$, we have $2a_2=a+1$.
We put $\lambda=(2,2a-2,a)$. Then $\ord_\lambda(g)=4a-1$, and we need $\ord_\lambda(I)\ge 2a$.

\begin{center}
\begin{tabular}{|c||c|c|c|c|c|c|}
\hline
$f$ & $x^a$ & $y^b$ & $z^2$ & $xy$ & $x^{a_2}z$ & $y^{b_3}z$ \\
\hline
$\ord_\lambda(f)$ & \cellcolor{mylightgray}$2a$ & $b(2a-2)$ & \cellcolor{mylightgray}$2a$ & \cellcolor{mylightgray}$2a$ & $2a+1$ & $b_3(2a-2)+a$ \\
\hline
\end{tabular}
\end{center}

Since $a,b\ge 3$ and $b_3\ge 2$, we have $\ord_\lambda(I)=2a$.

\medskip
\noindent\rule{\textwidth}{0.4pt}
\medskip

\noindent
\underline{$g \in \{x^{a_2-1}y^2z, x^{a_2-1}yz^2\}$}:

These monomials $g$ appear in Subcase 3-b, where we may assume $a_1 \ge a_2$ and $b_3 = 1$.
We divide into two cases: $a_1 + a_2 \le a$ and $a_1 + a_2 > a$.

\medskip

\noindent\textbf{When $a_1 + a_2 \le a$.}

We put $\lambda = (1,a_2, a_2)$. Then $\ord_\lambda(g) = 4a_2-1$, and we need $\ord_\lambda(I) \ge 2a_2$.

\begin{center}
\begin{tabular}{|c||c|c|c|c|c|c|}
\hline
$f$ & $x^a$ & $y^b$ & $z^c$ & $x^{a_1}y$ & $x^{a_2}z$ & $yz$ \\
\hline
$\ord_\lambda(f)$ & $a$ & $ba_2$ & $ca_2$ & $a_1 + a_2$ & $a_2 + a_2$ & \cellcolor{mylightgray}$2a_2$ \\
\hline
\end{tabular}
\end{center}

Since $a \ge a_1 + a_2 \ge 2a_2$, $b, c \ge 2$, we have $\ord_\lambda(I) = 2a_2$.

\medskip

\noindent\textbf{When $a_1 + a_2 > a$.}

In this case, $a$ is odd and $a_1 = a_2 = \frac{a+1}{2}$.

We put $\lambda = (2,a,a)$. Then $\ord_\lambda(g) = 4a-1$, and we need $\ord_\lambda(I) \ge 2a$.

\begin{center}
\begin{tabular}{|c||c|c|c|c|c|c|}
\hline
$f$ & $x^a$ & $y^b$ & $z^c$ & $x^{a_1}y$ & $x^{a_2}z$ & $yz$ \\
\hline
$\ord_\lambda(f)$ & \cellcolor{mylightgray}$2a$ & $ba$ & $ca$ & $2a + 1$ & $2a+1$ & \cellcolor{mylightgray}$2a$ \\
\hline
\end{tabular}
\end{center}

Since $b, c \ge 2$, we have $\ord_\lambda(I) = 2a$.

\medskip
\noindent\rule{\textwidth}{0.4pt}

\medskip

\noindent
\underline{$g = x^{a-1}yz$}:

This monomial $g$ appears only in Subcase 3-c, where $a$ is odd, $b = c = 2$, $a_1 = a_2 = \frac{a+1}{2}$, and $b_3 = 1$.
We put $\lambda = (2,a,a)$. Then $\ord_\lambda(g) = 4a-1$ and $\ord_\lambda(I) = 2a$ by a similar computation to the previous case.

\medskip
\noindent\rule{\textwidth}{0.4pt}

\medskip

\noindent
\underline{$g = x^{a_1+a_2-1}yz$}:

This monomial $g$ appears only in Subcase 3-b, where we have $b = c = 2$, $a_1 \ge a_2 \ge 2$, $b_3 = 1$, and $a_1 + a_2 \le a$.
We divide into two cases: $2a_1 \le a$ and $a_1 + a_2 \le a < 2a_1$.

\medskip

\noindent\textbf{When $2a_1 \le a$.}

We put $\lambda = (1, a_1, 2a_1 - a_2)$. Then $\ord_\lambda(g) = 4a_1-1$, and we need $\ord_\lambda(I) \ge 2a_1$.

\begin{center}
\begin{tabular}{|c||c|c|c|c|c|c|}
\hline
$f$ & $x^a$ & $y^2$ & $z^2$ & $x^{a_1}y$ & $x^{a_2}z$ & $yz$ \\
\hline
$\ord_\lambda(f)$ & $a$ & \cellcolor{mylightgray}$2a_1$ & $4a_1-2a_2$ & \cellcolor{mylightgray}$2a_1$ & $3a_1-a_2$ & \cellcolor{mylightgray}$2a_1$ \\
\hline
\end{tabular}
\end{center}

Since $a \ge 2a_1$, $a_2 \ge 2$, we have $\ord_\lambda(I) = 2a_1$.

\medskip

\noindent\textbf{When $a_1 + a_2 \le a < 2a_1$.}

In this case, $a$ is odd and $a_1 = \frac{a+1}{2}$. Hence $a_2 \le \frac{a-1}{2}$.

We put $\lambda = (2, a, 2(a - a_2))$. Then $\ord_\lambda(g) = 4a-1$, and we need $\ord_\lambda(I) \ge 2a$.

\begin{center}
\begin{tabular}{|c||c|c|c|c|c|c|}
\hline
$f$ & $x^a$ & $y^2$ & $z^2$ & $x^{a_1}y$ & $x^{a_2}z$ & $yz$ \\
\hline
$\ord_\lambda(f)$ & \cellcolor{mylightgray}$2a$ & \cellcolor{mylightgray}$2a$ & $4(a-a_2)$ & $2a + 1$ & \cellcolor{mylightgray}$2a$ & $3a - 2a_2$ \\
\hline
\end{tabular}
\end{center}

Since $2a_2 \le a-1$, we have $\ord_\lambda(I) = 2a$.

\medskip
\noindent\rule{\textwidth}{0.4pt}

\bigskip

Hence, for each $g$ in Table~\ref{tab:classification_6gen}, we have verified that $\ord_\lambda(g) < \ord_\lambda(I^2)$ for an appropriately chosen $\lambda$.
By Theorem~\ref{thm:RRV} and Proposition~\ref{prop:ord_test}, we conclude that $I$ is normal.

\subsection{The Case $\mu_A(I) =7$} \label{sec:7gens}

\subsubsection{Classification of integrally closed monomial ideals}

Let $I$ be an integrally closed monomial ideal with $\mu_A(I) = 7$ and $\height_AI = 3$. 
We may assume $x,y,z \notin I$, since otherwise we know that $I$ is normal similar to the case $\mu_A(I) = 6$.

By Lemma \ref{lem:key}, we can write
$I=(x^a, y^b, z^c, x^{a_1}y^{b_1}, x^{a_2}z^{c_2}, y^{b_3}z^{c_3}, f)$
for some $1 \le a_1, a_2 < a$, $1 \le b_1, b_3 < b$, $1 \le c_2, c_3 < c$, and monomial $f \in k[x,y,z]$.
Suppose $f \in (xyz)$.
Then $I_1 = I \cap k[x,y] = (x^a,y^b, x^{a_1}y^{b_1})$, $I_2 = I \cap k[x,z] = (x^a, z^c, x^{a_2}z^{c_2})$, $I_3 =I\cap k[y,z] = (y^b, z^c, y^{b_3}z^{c_3})$.
By an argument similar to that in the proof of Theorem \ref{thm:6gens}, after permuting $x,y,z$ if necessary, we have $b_3=c_3=1$.
Hence $yz \in I$, which implies that $f$ cannot be a minimal generator of $I$. This is a contradiction.
Therefore, after a suitable permutation if necessary, we may assume $f \in k[x,y]$.

Thus, by renaming the exponents, we can write
$$
I=(x^a, y^b, z^c, x^{a_1}y^{b_1}, x^{a_2}y^{b_2}, x^{a_3}z^{c_3}, y^{b_4}z^{c_4})
$$
for some positive integers $a_1, a_2, a_3, b_1, b_2, b_4, c_3, c_4 \ge 1$ satisfying $1 \le a_2 < a_1 < a$, $1 \le a_3 < a$, $1 \le b_1 < b_2 < b$, $1 \le b_4 < b$, and $1\le c_3, c_4 < c$.

\begin{thm}\label{thm:7gens}
	Let $a,b,c,a_1, a_2, a_3, b_1, b_2, b_4, c_3, c_4$ be positive integers with $1 \le a_2 < a_1 < a$, $1 \le a_3 < a$, $1 \le b_1 < b_2 < b$, $1 \le b_4 < b$, and $1\le c_3, c_4 < c$.
	Put $I=(x^a, y^b,z^c, x^{a_1} y^{b_1}, x^{a_2}y^{b_2}, x^{a_3}z^{c_3}, y^{b_4}z^{c_4})$.
	Then the following conditions are equivalent.
	
	\begin{enumerate}
		\item $I=\overline{I}$.
		\item $I_i = \overline{I_i}$ for every $i=1,2,3$, and one of $xz$ and $yz$ is in $I$.
		\item $c_3 = c_4=1$ and, up to permuting $x$ and $y$, we have $a_1=2$ and $a_2=1$, and one of the following holds:
		\begin{enumerate}
			\item $a,b \ge 4$, $b_1 = 1$, $2\le b_2 \le \ceil{\frac{b+1}{2}}$, $1\le a_3 \le \ceil{\frac{a}{2}}$, $1\le b_4 \le \ceil{\frac{b}{2}}$, and $a_3=1$ or $b_4 = 1$.
			\item $b=3$, $c=2$, $b_1=1$, $b_2=2$, $1 \le a_3 \le \ceil{\frac{a}{2}}$, $1\le b_4 \le 2$, and $a_3 = 1$ or $b_4 = 1$.
			\item $a=3$, $c=2$, $1 \le b_1 \le \ceil{\frac{b}{3}}$, $\max\{2b_1-1, b_1+1\} \le b_2 \le \ceil{\frac{b+b_1}{2}}$, $1\le a_3 \le 2$, $1 \le b_4 \le \ceil{\frac{b}{2}}$, and $a_3=1$ or $b_4=1$.
			\item $a=3$, $c\ge3$, $1\le b_1\le \left\lceil\frac b3\right\rceil$, $\max\{2b_1-1,b_1+1\}\le b_2\le \left\lceil\frac{b+b_1}{2}\right\rceil$, and $a_3=b_4=1$.
		\end{enumerate}
	\end{enumerate}	
\end{thm}

\begin{proof}
\textbf{(1) $\Rightarrow$ (2):}
Suppose (1). Then $I_i$ is integrally closed for every $i=1,2,3$. We first show that $c_3 = c_4=1$ and, up to permuting $x$ and $y$, $a_1=2$ and $a_2=1$.
If $c = 2$, then we get $c_3 = c_4=1$ automatically.
Suppose $c \ge 3$. 
Since $I_2 = (x^a, z^c, x^{a_3}z^{c_3})$ and $I_3 = (y^b, z^c, y^{b_4}z^{c_4})$ are integrally closed, 
applying Proposition~\ref{prop:2-variables} to these ideals with $a,b\ge3$ yields $a_3 = c_3 = b_4 = c_4=1$.
Since $I_1 = (x^a, y^b, x^{a_1}y^{b_1}, x^{a_2}y^{b_2})$ is integrally closed, again by Proposition \ref{prop:2-variables}, up to permuting $x$ and $y$, we have $a_1 = 2$ and $a_2= 1$.

We now show that $a_3 = 1$ or $b_4=1$.
We may assume $a_1 = 2$ and $a_2 = c_3 = c_4 = 1$.
Furthermore, we may assume $c=2$, because if $c \ge 3$, then $a_3=b_4=1$ by Proposition \ref{prop:2-variables}.
Assume for contradiction that $a_3 \ge 2$ and $b_4 \ge 2$.

If $b_2 < b_4$, then $(xy^{b_2-1}z)^2
= y^{2b_2-2-b_1}(x^2y^{b_1})z^2 \in I^2$,
which implies $xy^{b_2-1}z \in \overline{I} \setminus I$,
a contradiction.

Therefore, we may assume $b_2 \ge b_4$.
In this case, note that $x^2y^{b_1}$ is the only generator of $I$
that can divide $x^{a_3-1}y^{b_4-1}z$.
Thus, $x^{a_3-1}y^{b_4-1}z \in I$
if and only if $a_3 \ge 3$ and $b_4 \ge b_1+1$.

If $x^{a_3-1}y^{b_4-1}z \notin I$, then
$(x^{a_3-1}y^{b_4-1}z)^2
= (x^{a_3}z)(y^{b_4}z)x^{a_3-2}y^{b_4-2} \in I^2$,
so $x^{a_3-1}y^{b_4-1}z \in \overline{I}$, contradicting $I=\overline{I}$.

On the other hand, if $x^{a_3-1}y^{b_4-1}z \in I$, then
$(xy^{b_4-1}z)^2 = (x^2y^{b_1}) z^2 y^{2b_4-2-b_1} \in I^2$
while $xy^{b_4-1}z \notin I$, which leads to the impossible conclusion that $I \ne \overline{I}$.

Therefore, we must have $a_3 = 1$ or $b_4=1$.

\medskip

\textbf{(2) $\Rightarrow$ (1):}
This follows from Lemma \ref{lem:xy}.

\medskip

\textbf{(2) $\Leftrightarrow$ (3):}
This follows from Proposition \ref{prop:2-variables}, similarly to the proof of Theorem \ref{thm:6gens}.
More precisely, condition (2) states that $I_1, I_2, I_3$ are integrally closed and at least one of $xz$ and $yz$ is in $I$.
By Proposition \ref{prop:2-variables}, the integrally closedness of $I_1, I_2, I_3$ is equivalent to the conditions on the exponents listed in (3)(a), (3)(b), and (3)(c), depending on the values of $a, b, c$.
The additional constraint "$a_3=1$ or $b_4=1$" in (3) ensures that at least one of $xz$ and $yz$ is in $I$, as shown above.
\end{proof}

\subsubsection{Determining $(I^2:\fkm)\setminus I^2$}

Let $I$ be an integrally closed monomial ideal in $A=k[x,y,z]$ with $\height_AI = 3$ and $\mu_A(I) = 7$.
In this subsubsection, we put 
$$\calX = \left\{g \in (I^2:\fkm)\setminus I^2 ~\middle|~ g = x^\alpha y^\beta z^\gamma~(\alpha, \beta, \gamma \in \mathbb{N}) \right\}.$$

We use the notation in Theorem \ref{thm:7gens}. Then we always have $c_3 = c_4 =1$.
Furthermore, after switching $x$ and $y$ if necessary, we may assume $a_1=2$ and $a_2 = 1$.
Namely,
$$
I=(x^a, y^b, z^c, x^2y^{b_1}, xy^{b_2}, x^{a_3}z, y^{b_4}z).
$$
Here, no assumption is made on the ordering of $a,b,c$.

\noindent
\textbf{Case 1:}  $a_3=b_4=1$.

Since $xyz^2 \in I^2$, if $g = x^\alpha y^\beta z^\gamma \in (xyz)$ satisfies $g \notin I^2$, then $\gamma = 1$.
Hence, we need to consider the monomial $g = x^\alpha y^\beta z$.

\medskip

\noindent
\underline{Subcase 1-a:} $a = 3$.

In this case, we have $x^3yz \in I^2$, which implies $1 \le \alpha \le 2$.
It can be directly checked that 
$$x^2y^{b_1+1}z, xy^{b_2+1}z \in I^2 \quad \text{ and } \quad x^2y^{b_1}z, xy^{b_2}z \notin I^2.$$
Therefore, we have $\calX = \{x^2y^{b_1}z, xy^{b_2}z\}$.


\medskip

\noindent
\underline{Subcase 1-b:} $b = 3$.

We have $b_1 = 1$ and $b_2 = 2$. Then, by switching $x$ and $y$, we can reduce to Subcase 1-a.

\medskip

\noindent
\underline{Subcase 1-c:} $a,b \ge 4$.

Note that, in this case, we have $b_1 = 1$. 
Indeed, since $\mu_{k[x,y]}(I_1) = 4$ where $I_1 = I \cap k[x,y] = (x^a, y^b, x^2y^{b_1},  xy^{b_2})k[x,y]$, 
we have $\indeg(I_1) = 3$, which implies $b_1 = 1$ or $b_2 = 2$. 
Even if $b_2=2$, the condition $0<b_1<b_2$ forces $b_1 = 1$.
Hence $x^3yz \in I^2$, which implies $1 \le \alpha \le 2$.
Therefore, as in Subcase 1-a, we have $\calX = \{x^2y^{b_1}z, xy^{b_2}z\}$.

\medskip

This completes the determination of $\calX$ in Case 1. 
Although the analysis was divided into three subcases, we have verified that each reduces to the same form.

\medskip

\textit{Reduction Step.}
When $c\ge 3$, we always have $a_3 = b_4=1$. Thus, Case 1 handles all cases with $c \ge 3$, so from now on we may assume $c = 2$. Moreover, we may assume $a_3 \ge 2$ or $b_4 \ge 2$.
By Theorem \ref{thm:7gens}, we also have $a_3=1$ or $b_4=1$. Hence one of $a_3$ and $b_4$ is $1$ and the other is at least $2$.

\bigskip

\noindent
\textbf{Case 2:} $b=3$ and $c=2$.

In this case, $b_1=1$, $b_2 = 2$, and $1 \le b_4 \le 2$.

\medskip

\noindent
\underline{Subcase 2-a:} $b_4 = 2$.

By our assumption, $a_3 =1$ and hence $I=(x^a, y^3, z^2, x^2y, xy^2, xz, y^2z)$.
Then we can directly verify that $x^2yz, xy^2z, xyz^2 \in \calX$, and hence $$\calX = \{x^2yz, xy^2z, xyz^2\}.$$

\medskip

\noindent
\underline{Subcase 2-b:} $b_4= 1$.

Then $I=(x^a, y^3, z^2, x^2y, xy^2, x^{a_3}z, yz)$. 
In this case, we easily verify that $xy^2z, xyz^2 \in \calX$.
Moreover, since $x^{a_3+1}yz \notin I^2$ and $x^{a_3+2}yz = ( x^{a_3}z)(x^2y) \in I^2$, we have $$\calX = \{x^{a_3+1}yz, xy^2z, xyz^2\}.$$

\bigskip

\noindent
\textbf{Case 3:} $a, b\ge 4$ and $c=2$.

In this case, as in Subcase 1-c, we may assume $b_1 = 1$.
Furthermore, since $a \ge 4$ and $a_3 \le \ceil{\frac{a}{2}}$, we have $a_3 + 2 \le a$. 

We further divide this case into five subcases.

\medskip

\noindent
\underline{Subcase 3-a:} $b_4 = 1$.

Note that $I=(x^a, y^b, z^2, x^2y, xy^{b_2}, x^{a_3}z, yz)$.
Since $x^2y^2z, x^2yz^2, yz^3, y^2z^2 \in I^2$, if $g = x^\alpha y^\beta z^\gamma \in (xyz)$ satisfies $g \notin I^2$, then we have
\begin{itemize}
	\item $1\le \gamma \le  2$,
	\item if $\gamma = 2$, then $\alpha = \beta = 1$, and
	\item if $\gamma = 1$, then $\alpha = 1$ or $\beta = 1$.
\end{itemize}
Hence $g \in \{x^\alpha yz , xy^\beta z, xyz^2\}$ for some $\alpha, \beta \in \mathbb{N}$.

For the case $g = x^\alpha yz$, we have $\alpha = a_3 + 1$, since $a_3 + 2 \le a$.
For the case $g = xy^\beta z$, we can directly check that $\beta = b_2$.
These observations imply that $\calX = \{x^{a_3+1}yz, xy^{b_2}z, xyz^2\}$.

\medskip

\textit{Reduction Step.} 
From now on, we assume $a_3 = 1$. Hence $b_4 \ge 2$ and $$I=(x^a, y^b, z^2, x^2y, xy^{b_2}, xz, y^{b_4}z).$$
In particular, since $x^3yz, xz^3, x^2z^2 \in I^2$, if a monomial $g = x^\alpha y^\beta z^\gamma \in (xyz)$ satisfies $g \notin I^2$, then $\alpha \le 2$ and $\gamma \le 2$, and we cannot have $\alpha = \gamma = 2$ simultaneously.
Thus, we have $g \in \{x^2y^{\beta_1} z, xy^{\beta_2} z^2, xy^{\beta_3} z\}$ for some $\beta_1, \beta_2, \beta_3 \in \mathbb{N}$.
To determine $\calX$, we compute $\beta_1 = \max \{\beta \ge 1 \mid x^2y^\beta z \notin I^2\}$, $\beta_2 = \max \{\beta \ge 1 \mid xy^\beta z^2 \notin I^2\}$, and $\beta_3 = \max \{\beta \ge 1 \mid xy^\beta z \notin I^2\}$ in each subcase.

\medskip

\noindent
\underline{Subcase 3-b:} $b$ is odd and $b < b_2 + b_4$.

Recall that $b_2 \le \ceil{\frac{b+b_1}{2}} = \ceil{\frac{b+1}{2}}$ and $b_4 \le \ceil{\frac{b}{2}}$. Since $b$ is odd, we have $b_2 \le \frac{b+1}{2}$ and $b_4 \le \frac{b+1}{2}$.
Therefore, the assumption $b < b_2 + b_4$ implies that $b_2 = b_4 = \frac{b+1}{2}$.
Moreover, we can compute that $\beta_1 = \beta_2 = b_2-1$ and $\beta_3 = b-1$. 
Hence $$\calX = \{x^2y^{b_2-1}z, xy^{b-1}z, xy^{b_2-1}z^2\}$$ in this case.

\noindent
\underline{Subcase 3-c:} $b$ is even and $b < b_2 + b_4$.

Similarly to Subcase 3-b, we have $$\calX = \{x^2y^{b_2-1}z, xy^{b-1}z, xy^{b_4-1}z^2\}.$$ 
We nonetheless treat this as a separate case, since the values of $b_2$ and $b_4$
are determined differently from those in Subcase 3-b.
Note that since $b_2 \le \frac{b+2}{2}$ and $b_4 \le \frac{b}{2}$, the assumption $b < b_2+b_4$ implies $b_2 = \frac{b+2}{2}$ and $b_4 = \frac{b}{2}$.

\medskip

\noindent
\underline{Subcase 3-d:} $b\ge b_2+b_4$ and $b_2 \le b_4$.

Since $b \ge b_2+b_4$ (resp. $b_2 \le b_4$), we have $\beta_3 = b_2+b_4-1$ (resp. $\beta_1 = \beta_2 = b_2-1$).
Therefore, $$\calX = \{x^2y^{b_2-1}z, xy^{b_2-1}z^2, xy^{b_2+b_4-1}z\}.$$

\medskip

\noindent
\underline{Subcase 3-e:} $b\ge b_2+b_4$ and $b_2 > b_4$.

By computing $\beta_1,\beta_2, \beta_3$ similarly to Subcase 3-d, it is easy to see that $$\calX = \{x^2y^{b_4}z, xy^{b_4-1}z^2, xy^{b_2+b_4-1}z\}.$$

\bigskip

\noindent
\textbf{Case 4:} $a = 3$, $b\ge 4$, and $c=2$.

Recall that it suffices to consider the case where exactly one of
$a_3$ and $b_4$ equals $1$.

\medskip
\noindent
\underline{Subcase 4-a:} $b_4 = 1$.

In this case, we have $a_3 \ge 2$. Moreover, since $a=3$, we must have $a_3 = 2$.
Namely, $I=(x^3, y^b, z^2, x^2y^{b_1}, xy^{b_2}, x^2z, yz)$, and hence $x^3yz, yz^3, x^2yz^2 \in I^2$. 
Therefore, if $g = x^\alpha y^\beta z^\gamma \in (xyz)$ satisfies $g \notin I^2$, then $\alpha \le 2$ and $\gamma \le 2$, and we cannot have $\alpha = 2$ and $\gamma =2$ simultaneously. Thus $g \in \{x^2y^{\beta_1} z, xy^{\beta_2} z^2, xy^{\beta_3} z\}$ for some $\beta_1, \beta_2, \beta_3 \in \mathbb{N}$.
By determining the appropriate values of $\beta_i$ for which $g \in I^2:\fkm$, we can verify that $$\calX = \{x^2y^{b_1}z, xyz^2, xy^{b_2}z\}.$$

\medskip

\textit{Reduction Step.}
From now on, we assume $a_3 = 1$ and $b_4 \ge 2$, and divide this case into six subcases.
In this setting, $I=(x^3, y^b, z^2, x^2y^{b_1}, xy^{b_2}, xz, y^{b_4}z)$.
If $g = x^\alpha y^\beta z^\gamma \in (xyz)$ satisfies $g \notin I^2$, then since $x^2z^2, x^4z, xz^3\in I^2$, we have $\alpha \le 3$ and $\gamma \le 2$, and if $\gamma = 2$ then $\alpha =1$. Thus $g \in \{x^3y^{\beta_1} z, x^2 y^{\beta_2} z, x y^{\beta_3} z^2, xy^{\beta_4} z\}$ for some $\beta_1, \beta_2, \beta_3, \beta_4 \in \mathbb{N}$.
We compute $\beta_1 = \max \{\beta \ge 1 \mid x^3y^\beta z \notin I^2\}$, $\beta_2 = \max \{\beta \ge 1 \mid x^2y^\beta z \notin I^2\}$, $\beta_3 = \max \{\beta \ge 1 \mid xy^\beta z^2 \notin I^2\}$, and $\beta_4 = \max \{\beta \ge 1 \mid xy^\beta z \notin I^2\}$ in each subcase.

\medskip

\noindent
\underline{Subcase 4-b:} $b_2 < b_4$.

In this case, since $b_4 \le \ceil{\frac{b}{2}}$, we have $b_2 + b_4 < 2b_4 \le b+1$.
Moreover, since $b_2 < b_4 < b_1 + b_4$, we obtain $\beta_1 = b_1-1$, $\beta_2 = \beta_3 = b_2-1$, and $\beta_4 = b_2+b_4-1$. However, note that when $b_1 = 1$, we have $\beta_1 = 0$, which contradicts the definition of $\beta_1$.
Therefore,
$$
\calX = 
\begin{cases}
	\{x^3y^{b_1-1}z, x^2y^{b_2-1}z, xy^{b_2-1}z^2, xy^{b_2+b_4-1}z\} & \text{(if $b_1 \ge 2$)}\\
	\{x^2y^{b_2-1}z, xy^{b_2-1}z^2, xy^{b_2+b_4-1}z\} & \text{(if $b_1 = 1$)}.
\end{cases}
$$

\medskip
\textit{Note: } 
For the remaining cases, we only need to compute $\beta_i$ $(i=1,2,3,4)$ similarly, so we state only the conclusions.

\medskip

\noindent
\underline{Subcase 4-c:} $a_3 = 1$, $b_4 \ge 2$, $b_1 \le b_4 \le b_2$, and $b < b_2+b_4$. Notice that $b_2 \le b_4 + b_1$, because  $2 b_2 \le b+b_1+1 \le b_2 + b_4+b_1$.

$$
\calX = 
\begin{cases}
	\{x^3y^{b_1-1}z, x^2y^{b_2-1}z, xy^{b_4-1}z^2, xy^{b-1}z\} & \text{(if $b_1 \ge 2$)}\\
	\{x^2y^{b_2-1}z, xy^{b_4-1}z^2, xy^{b-1}z\} & \text{(if $b_1 = 1$)}.
\end{cases}
$$

\medskip

\noindent
\underline{Subcase 4-d:} $a_3 = 1$, $b_4 \ge 2$, $b_1 \le b_4 \le b_2$, $b \ge b_2+b_4$, and $b_2 \le b_1 + b_4$.

$$
\calX = 
\begin{cases}
	\{x^3y^{b_1-1}z, x^2y^{b_2-1}z, xy^{b_4-1}z^2, xy^{b_2+b_4-1}z\} & \text{(if $b_1 \ge 2$)}\\
	\{x^2y^{b_2-1}z, xy^{b_4-1}z^2, xy^{b_2+b_4-1}z\} & \text{(if $b_1 = 1$)}.
\end{cases}
$$

\medskip

\noindent
\underline{Subcase 4-e:} $a_3 = 1$, $b_4 \ge 2$, $b_1 \le b_4 \le b_2$, $b \ge b_2+b_4$, and $b_2 > b_1 + b_4$.

$$
\calX = 
\begin{cases}
	\{x^3y^{b_1-1}z, x^2y^{b_1+b_4-1}z, xy^{b_4-1}z^2, xy^{b_2+b_4-1}z\} & \text{(if $b_1 \ge 2$)}\\
	\{x^2y^{b_1+b_4-1}z, xy^{b_4-1}z^2, xy^{b_2+b_4-1}z\} & \text{(if $b_1 = 1$)}.
\end{cases}
$$

\medskip

\noindent
\underline{Subcase 4-f:} $a_3 = 1$, $b_4 \ge 2$, $b_4 < b_1$, $b \ge b_2+b_4$, and $b_2 \le b_1 + b_4$.

$$
\calX = \{x^3y^{b_4-1}z, x^2y^{b_2-1}z, xy^{b_4-1}z^2, xy^{b_2+b_4-1}z\}.
$$

\medskip

\noindent
\underline{Subcase 4-g:} $a_3 = 1$, $b_4 \ge 2$, $b_4 < b_1$, $b \ge b_2+b_4$, and $b_2 > b_1 + b_4$.

$$
\calX = \{x^3y^{b_4-1}z, x^2y^{b_1+b_4-1}z, xy^{b_4-1}z^2, xy^{b_2+b_4-1}z\}.
$$

\bigskip

We summarize the verification in this subsubsection in Tables \ref{tab:7gen_cases_params} and \ref{tab:7gen_cases_candidates}.
For Case 1, although the analysis was divided into three subcases, they all yield the same form of candidate monomials. 
Moreover, in Subcases 2-a and 2-b, we have $b_1=1$ and $b_2=2$, so $x^2yz$ and $xy^2z$ can be regarded as $x^2y^{b_1}z$ and $xy^{b_2}z$, respectively. In Table \ref{tab:7gen_cases_candidates}, we present them in this form.

\begin{table}[ht]
\centering
\caption{Classification of 7-generated integrally closed ideals: Parameter conditions}
\label{tab:7gen_cases_params}
\small
\begin{tabular}{|c|c|c|c|c|c|c|c|p{6cm}|}
\hline
\textbf{Case} & $a$ & $b$ & $c$ & $b_1$ & $b_2$ & $a_3$ & $b_4$ & \textbf{Additional conditions} \\
\hline
1 & -- & -- & -- & -- & -- & $1$ & $1$ & -- \\
\hline
2-a & -- & $3$ & $2$ & $1$ & $2$ & $1$ & $2$ & -- \\
\hline
2-b & $\ge 4$ & $3$ & $2$ & $1$ & $2$ & $\ge 2$ & $1$ & -- \\
\hline
3-a & $\ge 4$ & $\ge 4$ & $2$ & $1$ & -- & $\ge 2$ & $1$ & -- \\
\hline
3-b & $\ge 4$ & $\ge 4$, odd & $2$ & $1$ & -- & $1$ & $\ge 2$ & $b < b_2+b_4$ \quad $\therefore$ $b_2=b_4=\frac{b+1}{2}$ \\
\hline
3-c & $\ge 4$ & $\ge 4$, even & $2$ & $1$ & -- & $1$ & $\ge 2$ & $b < b_2+b_4$ \quad $\therefore$ $b_2=\frac{b+2}{2}$, $b_4=\frac{b}{2}$ \\
\hline
3-d & $\ge 4$ & $\ge 4$ & $2$ & $1$ & -- & $1$ & $\ge 2$ & $b \ge b_2+b_4$, $b_4 \ge b_2$ \\
\hline
3-e & $\ge 4$ & $\ge 4$ & $2$ & $1$ & -- & $1$ & $\ge 2$ & $b \ge b_2+b_4$, $b_4 < b_2$ \\
\hline
4-a & $3$ & $\ge 4$ & $2$ & -- & -- & $2$ & $1$ & -- \\
\hline
4-b & $3$ & $\ge 4$ & $2$ & -- & -- & $1$ & $\ge 2$ & $b_2 < b_4$ \quad $\therefore$ $b \ge b_2+b_4$\\
\hline
4-c & $3$ & $\ge 4$ & $2$ & -- & -- & $1$ & $\ge 2$ & $b_1 \le b_4 \le b_2$, $b < b_2+b_4$ \\
\hline
4-d & $3$ & $\ge 4$ & $2$ & -- & -- & $1$ & $\ge 2$ & $b_1 \le b_4 \le b_2$, $b \ge b_2+b_4$, $b_2 \le b_1+b_4$ \\
\hline
4-e & $3$ & $\ge 4$ & $2$ & -- & -- & $1$ & $\ge 2$ & $b_1 \le b_4 \le b_2$, $b \ge b_2+b_4$, $b_2 > b_1+b_4$ \\
\hline
4-f & $3$ & $\ge 4$ & $2$ & -- & -- & $1$ & $\ge 2$ & $b_4 < b_1$, $b \ge b_2+b_4$, $b_2 \le b_1+b_4$ \\
\hline
4-g & $3$ & $\ge 4$ & $2$ & -- & -- & $1$ & $\ge 2$ & $b_4 < b_1$, $b \ge b_2+b_4$, $b_2 > b_1+b_4$ \\
\hline
\end{tabular}
\end{table}

\begin{table}[ht]
\centering
\caption{Classification of 7-generated integrally closed ideals: Candidate monomials}
\label{tab:7gen_cases_candidates}
\begin{tabular}{|c|c|c|c|c|}
\hline
\textbf{Case} & $g_1$ & $g_2$ & $g_3$ & $g_4$ \\
\hline
1 & $x^2y^{b_1}z$ & $xy^{b_2}z$ & -- & -- \\
\hline
2-a & $x^2y^{b_1}z$ & $xyz^2$ & $xy^{b_2}z$ &  -- \\
\hline
2-b & $x^{a_3+1}yz$ & $xyz^2$  & $xy^{b_2}z$ & -- \\
\hline
3-a & $x^{a_3+1}yz$ & $xyz^2$ & $xy^{b_2}z$ & -- \\
\hline
3-b & $x^2y^{b_2-1}z$ & $xy^{b_2-1}z^2$ & $xy^{b-1}z$ & -- \\
\hline
3-c & $x^2y^{b_2-1}z$ & $xy^{b_4-1}z^2$ & $xy^{b-1}z$ & -- \\
\hline
3-d & $x^2y^{b_2-1}z$ & $xy^{b_2-1}z^2$ & $xy^{b_2+b_4-1}z$ & -- \\
\hline
3-e & $x^2y^{b_4}z$ & $xy^{b_4-1}z^2$ & $xy^{b_2+b_4-1}z$ & -- \\
\hline
4-a & $x^2y^{b_1}z$ & $xyz^2$ & $xy^{b_2}z$ & -- \\
\hline
4-b & $x^2y^{b_2-1}z$ & $xy^{b_2-1}z^2$ & $xy^{b_2+b_4-1}z$ & $x^3y^{b_1-1}z$ $(*)$ \\
\hline
4-c & $x^2y^{b_2-1}z$ & $xy^{b_4-1}z^2$ & $xy^{b-1}z$ & $x^3y^{b_1-1}z$ $(*)$ \\
\hline
4-d & $x^2y^{b_2-1}z$ & $xy^{b_4-1}z^2$ & $xy^{b_2+b_4-1}z$ & $x^3y^{b_1-1}z$ $(*)$ \\
\hline
4-e & $x^2y^{b_1+b_4-1}z$ & $xy^{b_4-1}z^2$ & $xy^{b_2+b_4-1}z$ & $x^3y^{b_1-1}z$ $(*)$ \\
\hline
4-f & $x^2y^{b_2-1}z$ & $xy^{b_4-1}z^2$ & $xy^{b_2+b_4-1}z$ & $x^3y^{b_4-1}z$ \\
\hline
4-g & $x^2y^{b_1+b_4-1}z$ & $xy^{b_4-1}z^2$ & $xy^{b_2+b_4-1}z$ & $x^3y^{b_4-1}z$ \\
\hline
\multicolumn{5}{l}{$(*)$ Appears only when $b_1 \ge 2$} \\
\end{tabular}
\end{table}

\subsubsection{Finding $\lambda =(p,q,r)$ with $\ord_\lambda(g) < \ord_\lambda(I^2)$}

We proceed as in Section~3.3.3 to construct $\lambda = (p,q,r) \in \mathbb{N}^3$ for each candidate monomial $g$ 
in Table~\ref{tab:7gen_cases_candidates} such that $\ord_\lambda(g) < \ord_\lambda(I^2)$ for every monomial ideal $I$ of the form appearing in the corresponding cases

Throughout this subsubsection, we consider monomial ideals of the form
\[
I=(x^a, y^b, z^c, x^2y^{b_1}, xy^{b_2}, x^{a_3}z, y^{b_4}z),
\]
which appear in the classification given in the previous subsubsection.
Here $a,b,c,b_1, b_2, a_3, b_4$ are arbitrary positive integers satisfying the conditions of the corresponding case.
Note the following.

\begin{itemize}
	\item Always $\max\{b_1+1, 2b_1-1\} \le b_2 \le \ceil{\frac{b+b_1}{2}}$.
	\item If $a = 3$, then $1 \le b_1 \le \ceil{\frac{b}{3}}$. 
	\item If $a \ge 4$, then $b_1 = 1$.
	\item If $c = 2$, then $1 \le a_3 \le \ceil{\frac{a}{2}}$ and $1 \le b_4 \le \ceil{\frac{b}{2}}$.
	\item If $c \ge 3$, then $a_3 = b_4 = 1$.
\end{itemize}

Due to the increased complexity of the parameter ranges in the 
7-generator case, we summarize the choice of $\lambda$ in Table \ref{tab:lambda_7gen} at the end of this section.

\medskip

\noindent
\underline{$g = x^2y^{b_1}z$}:

We divide into two cases: $b_1 = 1$ and $b_1 \ge 2$.

\medskip

\noindent\textbf{When $b_1=1$.}

We put $\lambda = (1,1,2)$. Then $\ord_\lambda(g) = 5$, and we need $\ord_\lambda(I) \ge 3$.

\begin{center}
\begin{tabular}{|c||c|c|c|c|c|c|c|}
\hline
$f$ & $x^a$ & $y^b$ & $z^c$ & $x^2y^{b_1}$ & $xy^{b_2}$ & $x^{a_3}z$ & $y^{b_4}z$ \\
\hline
$\ord_\lambda(f)$ & $a$ & $b$ & $2c$ & \cellcolor{mylightgray} $b_1 + 2=3$ & $b_2 + 1$ & $a_3 + 2$ & $b_4 + 2$\\
\hline
\end{tabular}
\end{center}

Since $a,b \ge 3$, $c\ge 2$, $b_2 \ge 2$, and $a_3, b_4 \ge 1$, we have $\ord_\lambda(I) \ge 3$.

\medskip

\noindent
\textbf{When $b_1\ge 2$.}

In this case, we always have $b_4 = 1$ by checking Tables \ref{tab:7gen_cases_params}, \ref{tab:7gen_cases_candidates}.

We put $\lambda = (3b_1-2, 3, 9b_1-9)$. Then $\ord_\lambda(g) = 18b_1 - 13$, and we need $\ord_\lambda(I) \ge 9b_1-6$.

\begin{center}
\begin{tabular}{|c||c|c|c|c|c}
\hline
$f$ & $x^a$ & $y^b$ & $z^c$ & $x^2y^{b_1}$ & $xy^{b_2}$ \\
\hline
$\ord_\lambda(f)$ & $a(3b_1-2)$ & $3b$ & $9c(b_1-1)$ & $9b_1-4$ & $3b_1+3b_2 - 2$ \\
\hline
\end{tabular}\hspace{6em}

\hspace{15em}
\begin{tabular}{c|c|}
\hline
$x^{a_3}z$ & $yz$ \\
\hline
$9b_1 + a_3(3b_1-2) -9$ & \cellcolor{mylightgray} $9b_1-6$\\
\hline
\end{tabular}
\end{center}

Since $a\ge 3$, $c\ge 2$, $a_3 \ge 1$, and $b_1 \ge 2$, we obtain $a(3b_1-2), 9c(b_1-1), 9b_1 + a_3(3b_1-2) -9 \ge 9b_1-6$.
Moreover, $b_1 \le \lceil\frac{b}{3}\rceil$ implies $3b_1 \le b + 2$, hence $3b \ge 9b_1-6$.
Finally, $b_2 \ge 2b_1-1$ yields
$$3b_1 + 3b_2 -2 \ge 9b_1 - 5 > 9b_1-6.$$
Therefore, $\ord_\lambda(I) = 9b_1 - 6$.

\medskip
\noindent\rule{\textwidth}{0.4pt}
\medskip

\noindent
\underline{$g = xy^{b_2}z$}:

This monomial appears in Case 1, Subcases 2-a, 2-b, 3-a, and 4-a.

We first treat the case $b_2=2$.

\medskip

\noindent\textbf{When $b_2 = 2$.}

In this case, we have $b_1 = 1$.
This handles all of Subcases 2-a and 2-b, as well as parts of Case 1 and Subcases 3-a, and 4-a.

We put $\lambda = (1,1,2)$. Then $\ord_\lambda(g) = 5$, and the resulting table coincides with that for $g = x^2y^{b_1}z$ when $b_1 = 1$, and we omit the table here.

\medskip

\textit{Reduction Step.}
Since Subcases 2-a and 2-b are now complete, it remains to consider Case 1 and Subcases 3-a and 4-a.
In all these cases, we have $b_4 = 1$.
For the remaining cases, distinguish the following three possibilities:

	(i) $b_2 \ge 2b_1 + 2$  (ii) $b_2 \le 2b_1 + 1$ and $2b \ge 3b_2-1$ (iii) $b_2 \le 2b_1+1$ and $2b\le 3b_2-2$

\medskip

\noindent\textbf{When $b_2 \ge 2b_1 + 2$.}

We put $\lambda = (2b_2-2b_1-1, 2, 4b_2-2b_1-4)$. Then $\ord_\lambda(g) = 8b_2-4b_1-5$, and we need $\ord_\lambda(I) \ge 4b_2-2b_1-2$.

\begin{tabular}{|c||c|c|c|c|c}
\hline
$f$ & $x^a$ & $y^b$ & $z^c$ & $x^2y^{b_1}$ & $xy^{b_2}$  \\
\hline
$\ord_\lambda(f)$ & $a(2b_2-2b_1-1)$ & $2b$ & $2c(2b_2-b_1-2)$ & \cellcolor{mylightgray}$4b_2-2b_1-2$ & $4b_2-2b_1-1$  \\
\hline
\end{tabular}\hspace{4em}

\hspace{13em}
\begin{tabular}{c|c|}
\hline
$x^{a_3}z$ & $yz$\\
\hline
$a_3(2b_2-2b_1-1) + (4b_2-2b_1-4)$ & \cellcolor{mylightgray}$4b_2-2b_1-2$\\
\hline
\end{tabular}

By the assumption that $b_2 \ge 2b_1 + 2$, we have $a(2b_2-2b_1-1) \ge 6b_2 - 6b_1 -3 \ge 4b_2-2b_1 +1$ and $2c(2b_2 -b_1 - 2) \ge 8b_2 - 4b_1 - 8 > 4b_2-4 \ge 4b_2 - 2b_1 - 2$. Here, we use $a \ge 3$ and $c \ge 2$ also.
Moreover, $b_2 \le \lceil\frac{b+b_1}{2}\rceil$ implies $2b \ge 4b_2 - 2b_1 - 2$.
Finally, $a_3 \ge 1$ and $b_2-b_1 \ge 1$ yield $(2a_3+4)b_2 - (2a_3+2)b_1 -(a_3+4) \ge 4b_2 - 2b_1 - 2$.
Therefore, $\ord_\lambda(I) = 4b_2-2b_1-2$.

\medskip

\noindent\textbf{When $b_2 \le 2b_1+1$ and $2b \ge 3b_2-1$.}

We put $\lambda = (b_2, 2, 3b_2-3)$. Then $\ord_\lambda(g) = 6b_2-3$, and we need $\ord_\lambda(I) \ge 3b_2-1$.

\begin{center}
\begin{tabular}{|c||c|c|c|c|c|c|c|}
\hline
$f$ & $x^a$ & $y^b$ & $z^c$ & $x^2y^{b_1}$ & $xy^{b_2}$ & $x^{a_3}z$ & $yz$ \\
\hline
$\ord_\lambda(f)$ & $ab_2$ & $2b$ & $3c(b_2-1)$ & $2b_2+2b_1$ & $3b_2$ & $(a_3+3)b_2-3$ & \cellcolor{mylightgray}$3b_2-1$ \\
\hline
\end{tabular}
\end{center}

Since $a \ge 3$, $c, b_2 \ge 2$, $a_3 \ge 1$, and the given conditions, we have $\ord_\lambda(I) = 3b_2 - 1$.

\medskip

\noindent\textbf{When $b_2 \le 2b_1+1$ and $2b \le 3b_2-2$.}

In this case, we claim that $b_1 = \lceil\frac{b}{3}\rceil$ and $b_2 = \lceil\frac{b+b_1}{2}\rceil$.
Indeed, if $b_1 < \lceil\frac{b}{3}\rceil$, then $3b_1 \le b-1$ and $2b_2 \le b+b_1+1$, which imply $6b_2 \le 4b+2 \le 6b_2-2$, a contradiction.
Similarly, if $b_2 < \lceil\frac{b+b_1}{2}\rceil$, then $2b_2 \le b+b_1-1$, which gives an impossible inequality $4b+4 \le 6b_2 \le 4b-1$.

We put $\lambda = (b, 3, 3b-3)$. Then $\ord_\lambda(g) = 4b+3b_2-3 \ge 6b-1$, and it suffices to have $\ord_\lambda(I) \ge 3b$.

\begin{center}
\begin{tabular}{|c||c|c|c|c|c|c|c|}
\hline
$f$ & $x^a$ & $y^b$ & $z^c$ & $x^2y^{b_1}$ & $xy^{b_2}$ & $x^{a_3}z$ & $yz$ \\
\hline
$\ord_\lambda(f)$ & $ab$ & \cellcolor{mylightgray}$3b$ & $3c(b-1)$ & $2b+3b_1$ & $b+3b_2$ & $a_3b+3b-3$ & \cellcolor{mylightgray}$3b$ \\
\hline
\end{tabular}
\end{center}

Using $b_1 = \lceil\frac{b}{3}\rceil$ and $b_2 = \lceil\frac{b+b_1}{2}\rceil$, we have $2b + 3b_1 \ge 3b$ and $b+3b_2 \ge 3b + 2 > 3b$.
The remaining values are easily verified to be at least $3b$. Hence $\ord_\lambda(I) = 3b$.

\medskip

\textit{Reduction Step.}
Since all candidate monomials $g$ appearing in Case~1 have been handled, 
we may assume $c = 2$ in what follows.

\medskip
\noindent\rule{\textwidth}{0.4pt}

\medskip

\noindent
\underline{$g = xyz^2$}:

We divide into two cases: $a_3=1$ and $a_3 \ge 2$.

\medskip
\noindent\textbf{When $a_3=1$.}

This occurs only in Subcase~2-a (see Table~\ref{tab:7gen_cases_candidates}),
where $b=3$, $b_1=1$, and $b_2=b_4=2$.

We put $\lambda = (3,2,3)$. Then $\ord_\lambda(g) = 11$, and we need $\ord_\lambda(I) \ge 6$.

\begin{center}
\begin{tabular}{|c||c|c|c|c|c|c|c|}
\hline
$f$ & $x^a$ & $y^3$ & $z^2$ & $x^2y$ & $xy^2$ & $xz$ & $yz^2$ \\
\hline
$\ord_\lambda(f)$ & $3a$ & \cellcolor{mylightgray}$6$ & \cellcolor{mylightgray}$6$ & $8$ & $7$ & \cellcolor{mylightgray}$6$ & $8$\\
\hline
\end{tabular}
\end{center}

Since $a \ge 3$, we have $\ord_\lambda(I) = 6$.

\medskip

\noindent\textbf{When $a_3\ge 2$.}

This occurs in Subcases~2-b, 3-a, and 4-a.
Since $a_3 \ge 2$, we have $b_4 = 1$ by Theorem \ref{thm:7gens}.

We put $\lambda = (2,3,3)$. Then $\ord_\lambda(g) = 11$, and we need $\ord_\lambda(I) \ge 6$.

\begin{center}
\begin{tabular}{|c||c|c|c|c|c|c|c|}
\hline
$f$ & $x^a$ & $y^b$ & $z^2$ & $x^2y^{b_1}$ & $xy^{b_2}$ & $x^{a_3}z$ & $yz$ \\
\hline
$\ord_\lambda(f)$ & $2a$ & $3b$ & \cellcolor{mylightgray}$6$ & $3b_1 + 4$ & $3b_2 + 2$ & $2a_3 + 3$ & \cellcolor{mylightgray}$6$\\
\hline
\end{tabular}
\end{center}

Since $a, b \ge 3$, $b_1 \ge 1$, and $b_2, a_3 \ge 2$, all values in the table
are at least $6$, and hence $\ord_\lambda(I) = 6$.

\medskip
\noindent\rule{\textwidth}{0.4pt}
\medskip

\noindent
\underline{$g = x^{a_3+1}yz$}:

We divide into two cases: $2a_3 = a+1$ and $2a_3 \le a$.
Note that $a_3 \le \lceil\frac{a}{2}\rceil$ implies $2a_3 \le a+1$.
This monomial appears in Subcases 2-b and 3-a, where we always have $a\ge 4$, $b_1=1$, $a_3 \ge 2$, and $b_4 = 1$.

\medskip

\noindent\textbf{When $2a_3 =a+1$.}

We put $\lambda = (2, 2a-4, a)$. Then $\ord_\lambda(g) = 4a -1$, and we need $\ord_\lambda(I) \ge 2a$.

\begin{center}
\begin{tabular}{|c||c|c|c|c|c|c|c|}
\hline
$f$ & $x^a$ & $y^b$ & $z^2$ & $x^2y$ & $xy^{b_2}$ & $x^{a_3}z$ & $yz$ \\
\hline
$\ord_\lambda(f)$ & \cellcolor{mylightgray}$2a$ & $b(2a-4)$ & \cellcolor{mylightgray}$2a$ & \cellcolor{mylightgray}$2a$ & $b_2(2a - 4) + 2$ & $2a + 1$ & $3a-4$\\
\hline
\end{tabular}
\end{center}

Since $a \ge 4$, $b\ge 3$, and $b_2 \ge 2$, we have $\ord_\lambda(I) = 2a$.

\medskip

\noindent\textbf{When $2a_3 \le a$.}

We put $\lambda = (1, 2a_3-2, a_3)$. Then $\ord_\lambda(g) = 4a_3 -1$, and we need $\ord_\lambda(I) \ge 2a_3$.

\begin{center}
\begin{tabular}{|c||c|c|c|c|c|c|c|}
\hline
$f$ & $x^a$ & $y^b$ & $z^2$ & $x^2y$ & $xy^{b_2}$ & $x^{a_3}z$ & $yz$ \\
\hline
$\ord_\lambda(f)$ & $a$ & $b(2a_3-2)$ & \cellcolor{mylightgray}$2a_3$ & \cellcolor{mylightgray}$2a_3$ & $b_2(2a_3 -2) + 2$ & \cellcolor{mylightgray}$2a_3$ & $3a_3 -2$\\
\hline
\end{tabular}
\end{center}

Since $a \ge 2a_3$, $a_3 \ge 2$, and $b_2 \ge 2$, we have $\ord_\lambda(I) = 2a_3$.

\medskip
\noindent\rule{\textwidth}{0.4pt}
\medskip

\noindent
\underline{$g = xy^{b_2-1}z^2$}:

We divide into two cases: $b \ge b_2 + b_4$ and $b < b_2 + b_4$.
This monomial appears in Subcases 3-b, 3-d, and 4-b, where we always have $b\ge 4$, $a_3 = 1$, and $b_4 \ge b_2$.

\medskip

\noindent\textbf{When $b \ge b_2 + b_4$.}

We put $\lambda = (b_2, 1, b_2)$. Then $\ord_\lambda(g) = 4b_2-1$, and we need $\ord_\lambda(I) \ge 2b_2$.

\begin{center}
\begin{tabular}{|c||c|c|c|c|c|c|c|}
\hline
$f$ & $x^a$ & $y^b$ & $z^2$ & $x^2y^{b_1}$ & $xy^{b_2}$ & $xz$ & $y^{b_4}z$ \\
\hline
$\ord_\lambda(f)$ & $ab_2$ & $b$ & \cellcolor{mylightgray}$2b_2$ & $2b_2+b_1$ & \cellcolor{mylightgray}$2b_2$ & \cellcolor{mylightgray}$2b_2$ & $b_2 + b_4$\\
\hline
\end{tabular}
\end{center}

Since $b\ge b_2 + b_4 \ge 2b_2$, $a \ge 4$, and $b_1 \ge 1$, we have $\ord_\lambda(I) = 2b_2$.

\medskip

\noindent\textbf{When $b < b_2 + b_4$.}

In this case, $b$ is odd and $b_2 = b_4 = \frac{b+1}{2}$ (See Subcase 3-b).

We put $\lambda = (b,2,b)$. Then $\ord_\lambda(g) = 4b - 1$, and we need $\ord_\lambda(I) \ge 2b$.

\begin{center}
\begin{tabular}{|c||c|c|c|c|c|c|c|}
\hline
$f$ & $x^a$ & $y^b$ & $z^2$ & $x^2y^{b_1}$ & $xy^{b_2}$ & $xz$ & $y^{b_4}z$ \\
\hline
$\ord_\lambda(f)$ & $ab$ & \cellcolor{mylightgray}$2b$ & \cellcolor{mylightgray}$2b$ & $2b+2b_1$ & $2b+1$ & \cellcolor{mylightgray}$2b$ & $2b+1$\\
\hline
\end{tabular}
\end{center}

Since $a \ge 4$ and $b_1 = 1$, we have $\ord_\lambda(I) = 2b$.

\medskip
\noindent\rule{\textwidth}{0.4pt}
\medskip

\noindent
\underline{$g = xy^{b_4-1}z^2$}:

This monomial appears in Subcases 3-c, 3-e, 4-c, 4-d, 4-e, 4-f, and 4-g.
In all these cases, we have $a_3 = 1$, $b_2 \ge b_4 \ge 2$, and $b \ge b_2 + b_4$.

We put $\lambda = (2b_4-1, 2, 2b_4-1)$. Then $\ord_\lambda(g) = 8b_4 -5$, and we need $\ord_\lambda(I) \ge 4b_4-2$.

\begin{center}
\begin{tabular}{|c||c|c|c|c|c|c|c|}
\hline
$f$ & $x^a$ & $y^b$ & $z^2$ & $x^2y^{b_1}$ & $xy^{b_2}$ & $xz$ & $y^{b_4}z$ \\
\hline
$\ord_\lambda(f)$ & $a(2b_4-1)$ & $2b$ & \cellcolor{mylightgray}$4b_4-2$ & $4b_4 + 2b_1 -2$ & $2b_4 + 2b_2 -1$ & \cellcolor{mylightgray}$4b_4-2$ & $4b_4-1$\\
\hline
\end{tabular}
\end{center}

Since $a \ge 3$, $b_4 \ge 2$, $b \ge b_2 + b_4$, $b_2 \ge b_4$, and $b_1 \ge 1$, we have $\ord_\lambda(I) = 4b_4-2$.

\medskip
\noindent\rule{\textwidth}{0.4pt}
\medskip

\noindent
\underline{$g = x^2y^{b_4}z$}:

This monomial appears only in Subcase 3-e.

We put $\lambda = (b_4, 1, b_4+1)$. Then $\ord_\lambda(g) = 4b_4 +1$, and we need $\ord_\lambda(I) \ge 2b_4+1$.

\begin{center}
\begin{tabular}{|c||c|c|c|c|c|c|c|}
\hline
$f$ & $x^a$ & $y^b$ & $z^2$ & $x^2y$ & $xy^{b_2}$ & $xz$ & $y^{b_4}z$ \\
\hline
$\ord_\lambda(f)$ & $ab_4$ & $b$ & $2b_4+2$ & \cellcolor{mylightgray}$2b_4 + 1$ & $b_4 + b_2$ & \cellcolor{mylightgray}$2b_4+1$ & \cellcolor{mylightgray}$2b_4+1$\\
\hline
\end{tabular}
\end{center}

In Subcase 3-e, we assume $a\ge 4$, $b \ge b_2 + b_4$, $b_1 = 1$, and $b_2 > b_4$. Hence we have $\ord_\lambda(I) = 2b_4 +1$.

\medskip
\noindent\rule{\textwidth}{0.4pt}
\medskip

\noindent
\underline{$g = x^3y^{b_1-1}z$}:

This monomial appears in Subcases 4-b, 4-c, 4-d, and 4-e when $b_1 \ge 2$.
Hence we may assume $a=3$, $a_3 = 1$, and $2 \le b_1 \le b_4$.

We put $\lambda = (3b_1 -2 , 3, 6b_1-4)$. Then $\ord_\lambda(g) = 18b_1-13$, and we need $\ord_\lambda(I) \ge 9b_1-6$.

\begin{center}
\begin{tabular}{|c||c|c|c|c|c|c|c|}
\hline
$f$ & $x^3$ & $y^b$ & $z^2$ & $x^2y^{b_1}$ & $xy^{b_2}$ & $xz$ & $y^{b_4}z$ \\
\hline
$\ord_\lambda(f)$ & \cellcolor{mylightgray}$9b_1-6$ & $3b$ & $12b_1-8$ & $9b_1-4$ & $3b_1+3b_2-2$ & \cellcolor{mylightgray}$9b_1-6$ & $6b_1 + 3b_4 - 4$\\
\hline
\end{tabular}
\end{center}

Since $b_1 \le \lceil\frac{b}{3}\rceil$, we have $3b_1 \le b+2$, hence $3b \ge 9b_1-6$.
Moreover, $b_2 \ge 2b_1 - 1$ yields $3b_1 + 3b_2-2 \ge 9b_1 - 5 > 9b_1-6$.
Finally, $b_1 \le b_4$ implies $6b_1 + 3b_4 -4 \ge 9b_1 -4 > 9b_1-6$.
Therefore, $\ord_\lambda(I) = 9b_1-6$.

\medskip
\noindent\rule{\textwidth}{0.4pt}
\medskip

\noindent
\underline{$g = x^3y^{b_4-1}z$}:

This monomial appears only in Subcases 4-f and 4-g, where we have $a=3$, $a_3 = 1$, $2 \le b_4 < b_1 < b_2$, and $b_2 + b_4 \le b$.

We put $\lambda = (3b_4-2, 3, 6b_4-4)$. Then $\ord_\lambda(g) = 18b_4 - 13$, and we need $\ord_\lambda(I) \ge 9b_4-6$.

\begin{center}
\begin{tabular}{|c||c|c|c|c|c|c|c|}
\hline
$f$ & $x^3$ & $y^b$ & $z^2$ & $x^2y^{b_1}$ & $xy^{b_2}$ & $xz$ & $y^{b_4}z$ \\
\hline
$\ord_\lambda(f)$ & \cellcolor{mylightgray}$9b_4-6$ & $3b$ & $12b_4-8$ & $6b_4 + 3b_1-4$ & $3b_2 + 3b_4-2$ & \cellcolor{mylightgray}$9b_4-6$ & $9b_4 - 4$\\
\hline
\end{tabular}
\end{center}

Since $b \ge b_2 + b_4$, $b_4 < b_1 < b_2$, and $b_2 \ge 2b_1 -1$, we have $\ord_\lambda(I) = 9b_4-6$.

\medskip
\noindent\rule{\textwidth}{0.4pt}
\medskip

\noindent
\underline{$g = xy^{b-1}z$}:

This monomial appears in Subcases 3-b, 3-c, and 4-c, where we have $a_3 = 1$, $b_4 \ge 2$, and $b < b_2 + b_4$.
We divide into two cases: $b_2 > b_4$ and $b_2 = b_4$.
Since $b_2 \ge b_4$ always holds in these cases, these two cases cover all possibilities.

\medskip

\noindent\textbf{When $b_2 > b_4$.}

We put $\lambda = (b-b_2 + 1, 1, b_2-1)$. Then $\ord_\lambda(g) = 2b-1$, and we need $\ord_\lambda(I) \ge b$.

\begin{center}
\begin{tabular}{|c||c|c|c|c|c|c|c|}
\hline
$f$ & $x^a$ & $y^b$ & $z^2$ & $x^2y^{b_1}$ & $xy^{b_2}$ & $xz$ & $y^{b_4}z$ \\
\hline
$\ord_\lambda(f)$ & $a(b-b_2+1)$ & \cellcolor{mylightgray}$b$ & $2b_2-2$ & $2(b-b_2+1)+b_1$ & $b+1$ & \cellcolor{mylightgray}$b$ & $b_4+b_2-1$ \\
\hline
\end{tabular}
\end{center}

Since $b_2 \le \lceil\frac{b+b_1}{2}\rceil$, we have $2b_2 \le b+b_1 + 1$, hence
$$2(b-b_2+1)+b_1 \ge b + 1.$$
Moreover, $2b_1 -1 \le b_2$ yields
$$a(b-b_2+1) \ge 3(b-b_2+1) \ge b.$$
Finally, the assumption $b < b_2 + b_4$ and $b_2 > b_4$ imply $2b_2 -2 \ge b_2 + b_4 -1 \ge b$.
Therefore, $\ord_\lambda(I) = b$.

\medskip

\noindent\textbf{When $b_2 = b_4$.}

We put $\lambda = (b, 2, b+1)$. Then $\ord_\lambda(g) = 4b-1$, and we need $\ord_\lambda(I) \ge 2b$.

\begin{center}
\begin{tabular}{|c||c|c|c|c|c|c|c|}
\hline
$f$ & $x^a$ & $y^b$ & $z^2$ & $x^2y^{b_1}$ & $xy^{b_2}$ & $xz$ & $y^{b_2}z$ \\
\hline
$\ord_\lambda(f)$ & $ab$ & \cellcolor{mylightgray}$2b$ & $2b+2$ & $2b+2b_1$ & $b+2b_2$ & $2b+1$ & $2b_2+(b+1)$ \\
\hline
\end{tabular}
\end{center}

Since $a \ge 3$ and $b < b_2+b_4 = 2b_2$, we have $\ord_\lambda(I) = 2b$.

\medskip
\noindent\rule{\textwidth}{0.4pt}
\medskip

\noindent
\underline{$g = x^2y^{b_1+b_4-1}z$}:

We divide into two cases: $b_1 \le b_4$ and $b_1 > b_4$.
This monomial appears in Subcases 4-e and 4-g, where we have $a=3$, $a_3 = 1$, $b \ge b_2 + b_4$, and $b_2 \ge b_1 + b_4$.

\medskip

\noindent\textbf{When $b_1 \le b_4$.}

We put $\lambda = (b_4, 1, b_1+b_4)$. Then $\ord_\lambda(g) = 2b_1 + 4b_4-1$, and we need $\ord_\lambda(I) \ge b_1 + 2b_4$.

\begin{center}
\begin{tabular}{|c||c|c|c|c|c|c|c|}
\hline
$f$ & $x^3$ & $y^b$ & $z^2$ & $x^2y^{b_1}$ & $xy^{b_2}$ & $xz$ & $y^{b_4}z$ \\
\hline
$\ord_\lambda(f)$ & $3b_4$ & $b$ & $2b_1+2b_4$ & \cellcolor{mylightgray}$b_1 + 2b_4$ & $b_2 + b_4$ & \cellcolor{mylightgray}$b_1 + 2b_4$ & \cellcolor{mylightgray}$b_1+2b_4$ \\
\hline
\end{tabular}
\end{center}

Since $b_1 \le b_4$, $b \ge b_2 + b_4$, and $b_2 \ge b_1 + b_4$, we have $\ord_\lambda(I) = b_1 + 2b_4$.

\medskip

\noindent\textbf{When $b_1 > b_4$.}

We put $\lambda = (3b_1 -2, 3, 9b_1-3b_4-6)$. Then $\ord_\lambda(g) = 18b_1 - 13$, and we need $\ord_\lambda(I) \ge 9b_1-6$.

\begin{center}
\begin{tabular}{|c||c|c|c|c|c|c|c|}
\hline
$f$ & $x^3$ & $y^b$ & $z^2$ & $x^2y^{b_1}$ & $xy^{b_2}$ & $xz$ & $y^{b_4}z$ \\
\hline
$\ord_\lambda(f)$ & \cellcolor{mylightgray}$9b_1-6$ & $3b$ & $18b_1-6b_4-12$ & $9b_1-4$ & $3b_1+3b_2-2$ & $12b_1-3b_4-8$ & \cellcolor{mylightgray}$9b_1-6$ \\
\hline
\end{tabular}
\end{center}

Since $b_1 \le \lceil\frac{b}{3}\rceil$, we have $3b_1 \le b+2$, hence $3b \ge 9b_1-6$.
Moreover, $2b_1 - 1 \le b_2$ yields $3b_1 + 3b_2 -2 \ge 9b_1 -5 > 9b_1-6$.
Finally, $b_1 > b_4$ implies $18b_1- 6b_4-12, 12b_1-3b_4-8 \ge 9b_1-6$.
Therefore, $\ord_\lambda(I) = 9b_1-6$.

\medskip
\noindent\rule{\textwidth}{0.4pt}
\medskip

\noindent
\underline{$g = x^2y^{b_2-1}z$}:

This monomial appears in Subcases 3-b, 3-c, 3-d, 4-b, 4-c, 4-d, and 4-f.
We first verify that $b_2 \le b_1 + b_4$ always holds in these cases.
In Subcases 4-d and 4-f, we already assume $b_2 \le b_1 + b_4$.
In Subcases 3-b, 3-c, 3-d, 4-b, we assume $b_2 \le b_4$, whence obviously $b_2 \le b_4 < b_1 + b_4$.
For Subcase 4-c, the assumption $b < b_2 + b_4$ yields $2b_2 \le b + b_1 + 1 \le b_2 + b_1 + b_4$, hence $b_2 \le b_1 + b_4$.

We divide into five cases.

\medskip

\noindent\textbf{When $b_2 \le b_4$.}

We put $\lambda = (b_2-1, 1, b_2)$. Then $\ord_\lambda(g) = 4b_2-3$, and we need $\ord_\lambda(I) \ge 2b_2-1$.

\begin{center}
\begin{tabular}{|c||c|c|c|c|c|c|c|}
\hline
$f$ & $x^a$ & $y^b$ & $z^2$ & $x^2y^{b_1}$ & $xy^{b_2}$ & $xz$ & $y^{b_4}z$ \\
\hline
$\ord_\lambda(f)$ & $a(b_2-1)$ & $b$ & $2b_2$ & $2b_2+b_1-1$ & \cellcolor{mylightgray}$2b_2-1$ & \cellcolor{mylightgray}$2b_2-1$ & $b_4+b_2$ \\
\hline
\end{tabular}
\end{center}

Since $b_4 \le \lceil\frac{b}{2}\rceil$, we have $2b_2 \le 2b_4 \le b+1$. Hence $\ord_\lambda(I) = 2b_2-1$ by $a\ge 3$ and $b_2 \le b_4$.

\medskip

\textit{Reduction Step.} In what follows, we assume $b_2 > b_4$.

\medskip

\noindent\textbf{When $b_2 > b_4$ and $b_2 \ge 2b_1$.}

We put $\lambda = (3b_2 - 3b_1 - 1, 3, 3b_2-2)$. Then $\ord_\lambda(g) = 12b_2 - 6b_1 - 7$, and we need $\ord_\lambda(I) \ge 6b_2-3b_1-3$.

\begin{tabular}{|c||c|c|c|c}
\hline
$f$ & $x^a$ & $y^b$ & $z^2$ & $x^2y^{b_1}$ \\
\hline
$\ord_\lambda(f)$ & $a(3b_2 - 3b_1 - 1)$ & $3b$ & $6b_2-4$ & $6b_2- 3b_1-2$  \\
\hline
\end{tabular}

\hspace{13em}
\begin{tabular}{c|c|c|}
\hline
$xy^{b_2}$ & $xz$ & $y^{b_4}z$\\
\hline
$6b_2-3b_1-1$ & \cellcolor{mylightgray}$6b_2-3b_1-3$ & $3b_2+3b_4-2$\\
\hline
\end{tabular}

Since $b_2 \le \lceil\frac{b+b_1}{2}\rceil$, we have $2b_2 \le b+b_1+1$, hence $3b \ge 6b_2 - 3b_1 -3$.
Moreover, $b_2 \le b_1 + b_4$ yields $3b_4 + 3b_2 -2 \ge 6b_2 - 3b_1 -2$.
Therefore, $\ord_\lambda(I) = 6b_2 - 3b_1 -3$.

\medskip

\textit{Reduction Step.} 
In what follows, we assume $b_2 \le 2b_1-1$.
Since $2b_1 -1 \le b_2$ holds by the original assumption, we have $b_2 = 2b_1 - 1$.
Combined with $b_2 \le b_1 + b_4$, this yields $b_1 -1 \le b_4$ and hence $b_2 \le 2b_4+1$.

\medskip

\noindent\textbf{When $b_4 < b_2 = 2b_1-1$ and $b \ge b_2 + b_4 + 1$.}

We put $\lambda = (b_2, 2, 2b_2+1)$. Then $\ord_\lambda(g) = 6b_2 -1$, and we need $\ord_\lambda(I) \ge 3b_2$.

\begin{center}
\begin{tabular}{|c||c|c|c|c|c|c|c|}
\hline
$f$ & $x^a$ & $y^b$ & $z^2$ & $x^2y^{b_1}$ & $xy^{b_2}$ & $xz$ & $y^{b_4}z$ \\
\hline
$\ord_\lambda(f)$ & $ab_2$ & $2b$ & $4b_2+2$ & $2b_2 + 2b_1$ & \cellcolor{mylightgray}$3b_2$ & $3b_2+1$ & $2b_2 + 2b_4 + 1$ \\
\hline
\end{tabular}
\end{center}

Since $b_2 = 2b_1 - 1$, we have $2b_2 + 2b_1 = 3b_2 +1$.
Moreover, $b_2 \le 2b_4 + 1$ yields $2b_2 + 2b_4 + 1 \ge 3b_2$.
Finally, $b \ge b_2 + b_4 + 1$ implies $2b \ge 3b_2 + 1$.
Therefore, $\ord_\lambda(I) = 3b_2$.

\medskip

\noindent\textbf{When $b_4 < b_2 = 2b_1-1$ and $b \le b_2 + b_4$.}

We put $\lambda = (b, 3, 3b_2)$. Then $\ord_\lambda(g) = 2b + 6b_2 - 3$, and we need $\ord_\lambda(I) \ge b+3b_2-1$.

\begin{center}
\begin{tabular}{|c||c|c|c|c|c|c|c|}
\hline
$f$ & $x^a$ & $y^b$ & $z^2$ & $x^2y^{b_1}$ & $xy^{b_2}$ & $xz$ & $y^{b_4}z$ \\
\hline
$\ord_\lambda(f)$ & $ab$ & $3b$ & $6b_2$ & $2b + 3b_1$ & \cellcolor{mylightgray}$b + 3b_2$ & \cellcolor{mylightgray}$b + 3b_2$ & $3b_2 + 3b_4$ \\
\hline
\end{tabular}
\end{center}

Since $b_1 \le \lceil\frac{b}{3}\rceil$, we have $3b_1 \le b+2$, which yields
$$2b + 3b_1 \ge 3b \ge b+6b_1-4 = b+3b_2-1$$
by $2b_1 -1 = b_2$.
Moreover, $b \le b_2 + b_4$ and $b_2 > b_4$ imply $6b_2 > 3b_2 + 3b_4 \ge 3b$.
Therefore, $\ord_\lambda(I) = b + 3b_2 > b+3b_2-1$.

\medskip
\noindent\rule{\textwidth}{0.4pt}
\medskip

\noindent
\underline{$g = xy^{b_2+b_4-1}z$}:

This monomial appears in Subcases 3-d, 3-e, 4-b, 4-d, 4-e, 4-f, and 4-g, where we always have $b\ge 4$, $a_3=1$, and $b \ge b_2 + b_4$.
We divide into eight cases.

\medskip

\noindent\textbf{When $b_2 < b_4$.}

We put $\lambda = (4b_4-2b_2-2, 2, 2b_4-1)$. Then $\ord_\lambda(g) = 8b_4-5$, and we need $\ord_\lambda(I) \ge 4b_4-2$.

\begin{center}
\begin{tabular}{|c||c|c|c|c|c}
\hline
$f$ & $x^a$ & $y^b$ & $z^2$ & $x^2y^{b_1}$ & $xy^{b_2}$ \\
\hline
$\ord_\lambda(f)$ & $a(4b_4-2b_2-2)$ & $2b$ & \cellcolor{mylightgray}$4b_4-2$ & $8b_4-4b_2+2b_1-4$ & \cellcolor{mylightgray}$4b_4-2$ \\
\hline
\end{tabular}
\hspace{5em}

\hspace{20em}
\begin{tabular}{c|c|}
\hline
$xz$ & $y^{b_4}z$\\
\hline
$6b_4-2b_2-3$ & $4b_4-1$\\
\hline
\end{tabular}
\end{center}

Since $a\ge 3$ and $b_2 < b_4$, we have
$$a(4b_4-2b_2-2) \ge 6b_4, \quad 8b_4-4b_2 + 2b_1 -4 \ge 4b_4 + 2b_1, \quad 6b_4 - 2b_2 -3 \ge 4b_4-1.$$
Moreover, $b_4 \le \lceil\frac{b}{2}\rceil$ yields $2b \ge 4b_4-2$.
Therefore, $\ord_\lambda(I) = 4b_4-2$.

\medskip

\textit{Reduction Step:} 
In what follows, we may assume $b_2 \ge b_4$.

\medskip

\noindent\textbf{When $b_2 \ge b_4 \ge b_2-1$.}

We put $\lambda = (b_2 + b_4 - 1, 2, b_2 + b_4)$. Then $\ord_\lambda(g) = 4b_2+4b_4-3$, and we need $\ord_\lambda(I) \ge 2b_2+2b_4-1$.

\begin{tabular}{|c||c|c|c|c}
\hline
$f$ & $x^a$ & $y^b$ & $z^2$ & $x^2y^{b_1}$  \\
\hline
$\ord_\lambda(f)$ & $a(b_2+b_4-1)$ & $2b$ & $2b_2+2b_4$ & $2b_1 + 2b_2+2b_4-2$   \\
\hline
\end{tabular}

\hspace{18em}
\begin{tabular}{c|c|c|}
\hline
$xy^{b_2}$ & $xz$ & $y^{b_4}z$\\
\hline
$3b_2+b_4-1$ & \cellcolor{mylightgray}$2b_2+2b_4-1$ & $b_2+3b_4$\\
\hline
\end{tabular}

Since $a\ge 3$, $b \ge b_2 + b_4$, and $b_2 \ge b_4 \ge b_2-1$, we have $\ord_\lambda(I) = 2b_2 + 2b_4 -1$.

\medskip

\textit{Reduction Step.} In what follows, we assume $b_2-2 \ge b_4$.

\medskip

\noindent\textbf{When $b_2 -2 \ge b_4$ and $a \ge 4$.}

Note that $b\ge 4$ holds for this $g$.
Hence $b_1 = 1$.

We put $\lambda = (2b_2-3, 2, 4b_2-2b_4-4)$. Then $\ord_\lambda(g) = 8b_2-9$, and we need $\ord_\lambda(I) \ge 4b_2-4$.

\begin{center}
\begin{tabular}{|c||c|c|c|c|c|c|c|}
\hline
$f$ & $x^a$ & $y^b$ & $z^2$ & $x^2y$ & $xy^{b_2}$ & $xz$ & $y^{b_4}z$ \\
\hline
$\ord_\lambda(f)$ & $a(2b_2-3)$ & $2b$ & $8b_2-4b_4-8$ & \cellcolor{mylightgray}$4b_2-4$ & $4b_2-3$ & $6b_2-2b_4-7$ & \cellcolor{mylightgray}$4b_2-4$ \\
\hline
\end{tabular}
\end{center}

Since $a \ge 4$ and $b_2 \ge 2$, we have $a(2b_2 - 3) \ge 4b_2 -4$.
Moreover, $b_2-2 \ge b_4$ yields $8b_2- 4b_4-8 \ge 4b_2$ and $6b_2-2b_4 -7 \ge 4b_2 -3$.
Finally, $b_2 \le \lceil\frac{b+1}{2}\rceil$ implies $2b \ge 4b_2 - 4$.
Therefore, $\ord_\lambda(I) = 4b_2-4$.

\medskip

\textit{Reduction Step.} 
In what follows, we assume $a=3$ and $b_2-2 \ge b_4$. 

\medskip

\noindent\textbf{When $b_2-2 \ge b_4$, $b_1 \le b_4$, and $b_2 \le b_1 + b_4$.}

We put $\lambda = (b_4, 1, b_2)$. Then $\ord_\lambda(g) = 2b_2+2b_4-1$, and we need $\ord_\lambda(I) \ge b_2+b_4$.

\begin{center}
\begin{tabular}{|c||c|c|c|c|c|c|c|}
\hline
$f$ & $x^3$ & $y^b$ & $z^2$ & $x^2y^{b_1}$ & $xy^{b_2}$ & $xz$ & $y^{b_4}z$ \\
\hline
$\ord_\lambda(f)$ & $3b_4$ & $b$ & $2b_2$ & $b_1+2b_4$ & \cellcolor{mylightgray}$b_2 + b_4$ & \cellcolor{mylightgray}$b_2 + b_4$ & \cellcolor{mylightgray}$b_2 + b_4$ \\
\hline
\end{tabular}
\end{center}

Recall that $b \ge b_2 + b_4$ always holds for this $g$.
Since $b_4 \ge b_1$ and $b_2 \le b_1 + b_4$, we have $3b_4 \ge b_2 + b_4$.
Moreover, $b_2 -2 \ge b_4$ yields $2b_2 \ge b_2 + b_4 + 2$.
Therefore, $\ord_\lambda(I) = b_2 + b_4$.

\medskip

\noindent\textbf{When $b_2-2 \ge b_4$, $b_1 \le b_4$ and $b_2 \ge b_1 + b_4 + 1$.}

Note that $b_2 \ge b_1 + b_4 + 1$ obviously implies $b_2 - 2 \ge b_4$.
Thus, in this case, the condition $b_2 - 2 \ge b_4$ is redundant.
Nevertheless, we keep this condition in the statement in order to make it clear
that the case division exhausts all possible situations.

We put $\lambda = (3b_2-3b_1-2, 3, 6b_2-3b_1-3b_4-4)$. Then $\ord_\lambda(g) = 12b_2-6b_1-9$, and we need $\ord_\lambda(I) \ge 6b_2-3b_1-4$.

\begin{tabular}{|c||c|c|c|c}
\hline
$f$ & $x^3$ & $y^b$ & $z^2$ & $x^2y^{b_1}$  \\
\hline
$\ord_\lambda(f)$ & $9b_2-9b_1-6$ & $3b$ & $12b_2-6b_1-6b_4-8$ & \cellcolor{mylightgray}$6b_2-3b_1-4$   \\
\hline
\end{tabular}

\hspace{13em}
\begin{tabular}{c|c|c|}
\hline
$xy^{b_2}$ & $xz$ & $y^{b_4}z$\\
\hline
$6b_2-3b_1-2$ & $9b_2-6b_1-3b_4-6$ & \cellcolor{mylightgray}$6b_2-3b_1-4$\\
\hline
\end{tabular}

Since $b_2 \le \ceil{\frac{b+b_1}{2}}$, we have $3b \ge 6b_2-3b_1-3$. 
Independently, using $b_1 \le b_4$ and $b_2 \ge b_1 + b_4 + 1$,
we easily verify that the remaining candidates in the table above
are all not less than $6b_2-3b_1-4$.
Therefore $\ord_\lambda(I) = 6b_2-3b_1-4$.

\medskip

\textit{Reduction Step.} 
We now consider the case $b_1 > b_4$, which we further divide into three subcases: $b_2 \ge 2b_1 + 1$, $b_2 = 2b_1$, and $b_2 = 2b_1-1$.

\medskip

\noindent\textbf{When $b_2-2 \ge b_4$, $b_1 > b_4$, and $b_2 \ge 2b_1 + 1$.}

We put $\lambda = (3b_2-3b_1-2, 3, 6b_2-3b_1-3b_4-4)$ as in the previous case.
Hence $\ord_\lambda(g) = 12b_2 - 6b_1-9$, and the table is identical to the one above.

Since $b_2 \ge 2b_1 +1$, we have $9b_2 - 9b_1 -6 \ge 6b_2-3b_1-3 > 6b_2-3b_1-4$.
Similarly, $12b_2-6b_1 -6b_4 -8 > 6b_2-3b_1-4$.
Moreover, $b_2 \le \lceil\frac{b+b_1}{2}\rceil$ yields $3b \ge 6b_2-3b_1-3$.
Therefore, $\ord_\lambda(I) = 6b_2-3b_1-4$.

\medskip

\noindent\textbf{When $b_2-2 \ge b_4$, $b_1 > b_4$, and $b_2 = 2b_1$.}

First, we note that $b \ge 3b_1-1$, since $b_2 \le \ceil{\frac{b+b_1}{2}}$ and $b_2 = 2b_1$.

We put $\lambda = (3b_1 - 1, 3, 9b_1-3b_4-3)$. Then $\ord_\lambda(g) = 18b_1-7$, and we need $\ord_\lambda(I) \ge 9b_1-3$.

\begin{center}
\begin{tabular}{|c||c|c|c|c|c|c|c|}
\hline
$f$ & $x^3$ & $y^b$ & $z^2$ & $x^2y^{b_1}$ & $xy^{b_2}$ & $xz$ & $y^{b_4}z$ \\
\hline
$\ord_\lambda(f)$ & \cellcolor{mylightgray}$9b_1-3$ & $3b$ & $18b_1-6b_4-6$ & $9b_1-2$ & $9b_1-1$ & $12b_1 -3b_4 - 4$ & \cellcolor{mylightgray}$9b_1-3$ \\
\hline
\end{tabular}
\end{center}

Using the assumption $b_1 > b_4$ and the fact $b \ge 3b_1-1$, we easily verify $\ord_\lambda(I) = 9b_1 - 3$.

\medskip

\noindent\textbf{When $b_1 > b_4$ and $b_2 = 2b_1-1$.}

In this case, $b_2 = 2b_1-1 \ge b_1 + b_4$.

We put $\lambda = (b_2, 2, 3b_2-2b_4-1)$. Then $\ord_\lambda(g) = 6b_2-3$, and we need $\ord_\lambda(I) \ge 3b_2-1$.

\begin{center}
\begin{tabular}{|c||c|c|c|c|c|c|c|}
\hline
$f$ & $x^3$ & $y^b$ & $z^2$ & $x^2y^{b_1}$ & $xy^{b_2}$ & $xz$ & $y^{b_4}z$ \\
\hline
$\ord_\lambda(f)$ & $3b_2$ & $2b$ & $6b_2-4b_4-2$ & $2b_1+2b_2$ & $3b_2$ & $4b_2-2b_4-1$ & \cellcolor{mylightgray}$3b_2-1$ \\
\hline
\end{tabular}
\end{center}

Since $b_1 = \lceil\frac{b}{3}\rceil$, we have $2b \ge 6b_1 -4 = 3b_2 -1$.
Moreover, $b_2 \ge b_1+b_4 > 2b_4$ yields $6b_2 - 4b_4 - 2 \ge 4b_2$ and $4b_2-2b_4-1 \ge 3b_2$.
Therefore, $\ord_\lambda(I) = 3b_2-1$.

\bigskip

Consequently, for every candidate $g\in\mathcal{X}$ we have constructed a weight $\lambda$ with $\ord_\lambda(g)<\ord_\lambda(I^2)$. Hence $I$ is normal by Proposition~\ref{prop:ord_test} and Theorem~\ref{thm:RRV}.
This completes the proof of Theorem \ref{maintheorem}.

Table~\ref{tab:lambda_7gen} summarizes the valuations $\lambda$ used for each candidate monomial in the 7-generator case.

\begin{table}[h]
\centering
\scriptsize
\caption{Choice of $\lambda$ for each candidate monomial $g$}
\label{tab:lambda_7gen}
\renewcommand{\arraystretch}{1.3}
\begin{tabular}{|l|>{\raggedright\arraybackslash}p{2.5cm}|
                >{\raggedright\arraybackslash}p{4cm}|l|}
\hline
$g$ & Cases & Condition & $\lambda$ \\
\hline
$x^2y^{b_1}z$
& 1, 2-a, 4-a
& $b_1 = 1$
& $(1,1,2)$ \\
\cline{3-4}
& 
& $b_1 \ge 2$
& $(3b_1-2,3,9b_1-9)$ \\
\hline
$xy^{b_2}z$
& 1, 2-a, 2-b, 3-a, 
& $b_2 = 2$
& $(1,1,2)$ \\
\cline{3-4}
& 4-a
& $b_2 \ge 2b_1+2$
& $(2b_2-2b_1-1,2,4b_2-2b_1-4)$ \\
\cline{3-4}
&
& $2b_1-1 \le b_2 \le 2b_1+1$, $2b \ge 3b_2-1$
& $(b_2,2,3b_2-3)$ \\
\cline{3-4}
&
& $2b_1-1 \le b_2 \le 2b_1+1$, $2b \le 3b_2-2$
& $(b,3,3b-3)$ \\
\hline
$xyz^2$
& 2-a, 2-b, 3-a, 4-a
& $a_3 = 1$
& $(3,2,3)$\\
\cline{3-4}
& 
& $a_3 \ge 2$
& $(2,3,3)$ \\
\hline
$x^{a_3+1}yz$
& 2-b, 3-a
& $2a_3=a+1$
& $(2,2a-4,a)$ \\
\cline{3-4}
&
& $2a_3 \le a$
& $(1, 2a_3-2, a_3)$ \\
\hline
$xy^{b_2-1}z^2$
& 3-b, 3-d, 4-b
& $b \ge b_2 + b_4$
& $(b_2, 1, b_2)$ \\
\cline{3-4}
&
& $b < b_2 + b_4$ $\therefore$ $b_2=b_4=\frac{b+1}{2}$
& $(b, 2, b)$ \\
\hline
$xy^{b_4-1}z^2$
& 3-c, 3-e, 4-c, 4-d, 4-e, 4-f, 4-g
& --
& $(2b_4-1, 2,2b_4-1)$\\
\hline
$x^2y^{b_4}z$
& 3-e
& --
& $(b_4, 1, b_4+1)$ \\
\hline
$x^3y^{b_1-1}z$
& 4-b, 4-c, 4-d, 4-e
& --
& $(3b_1-2, 3, 6b_1-4)$\\
\hline
$x^3y^{b_4-1}z$ 
& 4-f, 4-g
& --
& $(3b_4-2, 3, 6b_4-4)$\\
\hline
$xy^{b-1}z$
& 3-b, 3-c, 4-c
& $b_2 > b_4$ 
& $(b-b_2+1, 1, b_2-1)$ \\
\cline{3-4}
&
& $b_2 = b_4$
& $(b, 2, b+1)$ \\
\hline
$x^2y^{b_1+b_4-1}z$
& 4-e, 4-g
& $b_1 \le b_4$
& $(b_4, 1, b_1 +b_4)$\\
\cline{3-4}
&
& $b_1 > b_4$
& $(3b_1-2, 3, 9b_1-3b_4-6)$\\
\hline
$x^2y^{b_2-1}z$
& 3-b, 3-c, 3-d, 4-b,
& $b_2 \le b_4$
& $(b_2-1, 1, b_2)$\\
\cline{3-4}
& 4-c, 4-d, 4-f
& $b_2 > b_4$, $b_2 \ge 2b_1$
& $(3b_2-3b_1-1, 3, 3b_2-2)$\\
\cline{3-4}
& 
& $b_4 < b_2 < 2b_1$, $b \ge b_2+ b_4 +1$
& $(b_2, 2, 2b_2+1)$\\
\cline{3-4}
& 
& $b_4 < b_2 < 2b_1$, $b \le b_2 + b_4$
& $(b,3,3b_2)$\\
\hline
$xy^{b_2+b_4-1}z$
& 3-d, 3-e, 4-b, 4-d
& $b_2 < b_4$
& $(4b_4 - 2b_2-2, 2, 2b_4-1)$\\
\cline{3-4}
& 4-e, 4-f, 4-g
& $b_2 \ge b_4 \ge b_2-1$
& $(b_2+b_4-1,2, b_2 + b_4)$\\
\cline{3-4}
&
& $b_2 -2 \ge b_4$, $a\ge 4$
& $(2b_2-3, 2, 4b_2-2b_4-4)$ \\
\cline{3-4}
&
& $b_2 -2\ge b_4$, $a=3$, $b_2 \le b_1 + b_4$
& $(b_4, 1, b_2)$ \\
\cline{3-4}
&
& $b_2 -2\ge b_4$, $a=3$, $b_2 \ge b_1 + b_4 + 1$
& $(3b_2-3b_1-2, 3, 6b_2-3b_1-3b_4-4)$ \\
\cline{3-4}
&
& $b_2-2 \ge b_4$, $a=3$, $b_1 > b_4$, $b_2 \ge 2b_1 + 1$
& $(3b_2-3b_1-2, 3, 6b_2-3b_1-3b_4-4)$ \\
\cline{3-4}
&
& $b_2-2 \ge b_4$, $a=3$, $b_1 > b_4$, $b_2 = 2b_1$
& $(3b_1-1, 3, 9b_1-3b_4-3)$\\
\cline{3-4}
&
& $b_2-2 \ge b_4$, $a=3$, $b_1 > b_4$, $b_2 \le 2b_1 -1$
& $(b_2, 2, 3b_2-2b_4-1)$\\
\hline
\end{tabular}
\end{table}

\section{Further consequences and structures}\label{sec:consequences}
\subsection{Extension to higher dimensions}\label{sec:higher_dim}

We show that our main result extends to higher dimensions for integrally closed monomial ideals with few generators.

\begin{theorem}\label{thm:higher_dim}
	Let $A = k[x_1,\ldots,x_d]$ be the polynomial ring with $d \ge 3$.
	Let $I$ be an integrally closed monomial ideal with $\height_AI = d$.
	If $\mu_A(I) \leq d+4$, then $I$ is normal.
\end{theorem}

\begin{proof}
When $d=3$, we already have the statement by Theorem \ref{maintheorem}. Hence we may assume $d \ge 4$.

Now, assume that $\mu_A(I) \le d+4$ and $x_i \notin I$ for all $1 \le i \le d$.
By Lemma~\ref{lem:key}, for each pair $i \ne j$ there exists a minimal generator of $I$ of the form $x_i^\ell x_j^m$ with $\ell,m \ge 1$.
Moreover, any minimal generating set of $I$ must contain the $d$ pure powers $x_1^{a_1},\ldots,x_d^{a_d}$.
Thus $\mu_A(I) \ge \binom{d}{2} + d = \frac{d(d+1)}{2}$.
Since $d \geq 4$, we have $\frac{d(d+1)}{2}  > d+4$, 
which contradicts our assumption.
Therefore at least one of $x_1, \ldots, x_d$ is in $I$.
Without loss of generality, we may assume $x_d \in I$.

Then $I/(x_d)$ is an integrally closed monomial ideal in $A/(x_d) \cong k[x_1,\ldots,x_{d-1}]$ 
with $\mu_{A/(x_d)}(I/(x_d)) \leq (d-1)+4$.
Iterating this reduction, we eventually reach the three-variable case with at most $7$ generators, and Theorem~\ref{maintheorem} applies.
\end{proof}

\subsection{Monomial parameter ideals and their integral closures}\label{sec:monomialparameter}

In this subsection, we examine integrally closed monomial ideals arising as integral closures of monomial parameter ideals, and we revisit a standard counterexample to normality.
The following proposition is well-known; we include an elementary proof for completeness.

\begin{proposition}\label{prop:formula}
Let $A = k[x_1,\ldots,x_d]$ be a polynomial ring.
For $a_1, \ldots,a_d \in \mathbb{N}$, set
$$
Q = (x_1^{a_1}, x_2^{a_2}, \ldots, x_d^{a_d}) \subseteq A.
$$
Then the integral closure of $Q$ is given by
$$
\overline{Q}  =
\left(
x_1^{b_1} x_2^{b_2} \cdots x_d^{b_d}~\middle|~\sum_{i=1}^d \frac{b_i}{a_i} \ge 1
\right).
$$
\end{proposition}

\begin{proof}
Let $x^\beta=x_1^{b_1}\cdots x_d^{b_d}$ be a monomial such that
$\sum_{i=1}^d \frac{b_i}{a_i}\ge 1$.
Set $m=\prod_{i=1}^d a_i$.
Then
$$
(x^\beta)^m = x_1^{mb_1}\cdots x_d^{mb_d} \in Q^m,
$$
and hence $x^\beta\in\overline{Q}$ by definition of integral closure.

\smallskip

Conversely, let $x^\beta=x_1^{b_1}\cdots x_d^{b_d}$ be a monomial in $\overline{Q}$.
By definition, there exist an integer $n\ge1$ and elements $f_i\in Q^i$ such that
$$
(x^\beta)^n + f_1(x^\beta)^{n-1} + \cdots + f_n = 0.
$$
Rewriting this equality, we obtain
$$
(\ast) \qquad (x^\beta)^n = -\sum_{i=1}^n f_i (x^\beta)^{n-i}.
$$

Each $f_i$ is a $k$-linear combination of monomials in $Q^i$, hence the right-hand side of $(\ast)$ is a $k$-linear combination of monomials.
Since monomials form a $k$-basis of $A$, the equality $(\ast)$ forces the monomial $(x^\beta)^n$ to occur on the right-hand side with a nonzero coefficient.
Therefore, there exists $1 \le i \le n$ and a monomial term $u \in Q^i$ such that
\[
u\,(x^\beta)^{\,n-i} = (x^\beta)^n .
\]
Equivalently, $(x^\beta)^i = u \in Q^i$.

Thus we can write
$$
i\beta = \sum_{j=1}^d m_j a_j e_j + \xi,
$$
where $m_j\in\mathbb{N}_0$ satisfy $\sum_{j=1}^d m_j = i$,
$\xi\in\mathbb{N}_0^d$, and $e_j$ denotes the $j$-th standard basis vector, where $\mathbb{N}_0$ denotes the set of non-negative integers.
Consequently, for each $j$ we have $i b_j \ge m_j a_j$, and hence
$$
\sum_{j=1}^d \frac{b_j}{a_j} \ge \sum_{j=1}^d \frac{m_j}{i} = 1.
$$

This completes the proof.
\end{proof}

Note that if $I=\overline{(x_1^{a_1}, x_2^{a_2}, \ldots, x_d^{a_d})}$, then $(x_1^{a_1}, x_2^{a_2}, \ldots, x_d^{a_d})$ is a minimal reduction of $I$.
Conversely, if an integrally closed monomial ideal $I$ contains a monomial parameter ideal $Q$ as a reduction, then $I=\overline{Q}$.

\begin{remark}
Proposition~\ref{prop:formula} can be interpreted geometrically
in terms of the Newton polyhedron of the monomial ideal
$(x_1^{a_1},\ldots,x_d^{a_d})$.
Indeed, the condition
$\sum_{i=1}^d \frac{b_i}{a_i} \ge 1$
describes precisely the supporting hyperplane of the Newton polyhedron.
\end{remark}

\begin{observation}
Let $A=k[x,y,z]$ and let $Q=(x^a,y^b,z^c)$ with $a\ge b\ge c\ge 2$.
Then, a direct check via Proposition~\ref{prop:formula} shows that,
for $I=\overline{Q}$, we have $\mu_A(I)=6$ if and only if $b=c=2$.
On the other hand, if $b\ge 3$, then monomials involving all three variables necessarily appear
among the minimal generators of $I$, and hence $\mu_A(I)\ge 8$.
In particular, $\mu_A(I)=7$ never occurs for ideals of the form
$I=\overline{(x^a,y^b,z^c)}$.

Now assume $a\ge 2$ and put $I=\overline{(x^a,y^2,z^2)}$.
Setting $d=\ceil{\frac{a}{2}}$, we have
\[
I=(x^a,y^2,z^2,x^d y,x^d z,yz).
\]
In this case, it is straightforward to verify that
$\overline{(x^{2a},y^4,z^4)}=I^2$.
Hence $\overline{I^2}=I^2$, and therefore $I$ is normal by
Theorem~\ref{thm:RRV}.

Thus, when restricting to integrally closed ideals obtained as the
integral closure of monomial parameter ideals, every $6$-generated
example is automatically normal.
\end{observation}

As already mentioned in Section~1, there exists a counterexample to the normality of integrally closed monomial ideals with $\mu_A(I)=8$ in $A=k[x,y,z]$.
In fact, it arises in a particularly tractable situation where $I$ admits a monomial parameter ideal as a reduction.
To illustrate why such a phenomenon can occur once we reach $8$ generators, let us revisit the following standard example.

\begin{example}\label{ex:counterexample}
The bound in Theorem~\ref{thm:higher_dim} is sharp.
Indeed, the well-known counterexample
$I=\overline{(x^7,y^3,z^2)}$
mentioned in Section~1 satisfies $\mu_A(I)=8$ and is not normal.
\end{example}

\begin{proof}
	The standard computation by using Proposition \ref{prop:formula} shows that $$I = (x^7, y^3, z^2, x^5y, x^3y^2, x^4z, y^2z, x^2yz).$$
	Therefore $x^6y^2z \notin I^2$.
	However, $(x^6y^2z)^2 = (x^5y) x^7 y^3 z^2 \in I^4 = (I^2)^2$ implies that $x^6y^2z \in \overline{I^2}$.
	Thus, $\overline{I^2} \ne I^2$.
	The eight monomials listed above form the minimal system of monomial generators of $I$. Hence $\mu_A(I) = 8$.
\end{proof}

\subsection{Rees algebras and reduction numbers of normal monomial ideals}\label{sec:reduction_numbers}

Throughout this subsection, we assume that $k$ is an infinite field, and we consider the localization $A_\fkm$ of $A$ at $\fkm$ where $A=k[x,y,z]$ and $\fkm = (x,y,z)$.

Let $I$ be a normal monomial ideal in $A$ with $\height_AI = 3$.
Since the Rees algebra $\calR(I_\fkm) = \bigoplus_{n \ge 0} (I_\fkm)^n$ is a normal affine semigroup ring,
by Hochster's theorem, $\calR(I_\fkm)$ is Cohen-Macaulay.

By a result of Goto and Shimoda \cite{GS}, the Cohen-Macaulayness of $\calR(I_\fkm)$ implies that the associated graded ring $\gr_{I_\fkm}(A_\fkm) = \bigoplus_{n \geq 0} (I_\fkm)^n/(I_\fkm)^{n+1}$ 
is also Cohen-Macaulay and its $a$-invariant is negative.

Since $k$ is infinite, $I_\fkm$ admits a minimal reduction $Q$ so that $Q$ is a parameter ideal of $A_\fkm$.
As $\gr_{I_\fkm}(A_\fkm)$ is Cohen-Macaulay, we have
$$
a(\gr_{I_\fkm}(A_\fkm))=\red_Q(I_\fkm)-3,
$$
where $\red_Q(I_\fkm) = \min\{n \geq 0 \mid (I_\fkm)^{n+1} = Q(I_\fkm)^n\}$ is the reduction number of $I_\fkm$ with respect to $Q$.
Hence $\red_Q(I_\fkm) \le 2$, since $a(\gr_{I_\fkm}(A_\fkm))$ is negative.
The same conclusion holds for the ideal $I$ in the polynomial ring $A$.

Therefore, we obtain the following as a corollary of Theorem \ref{maintheorem}.

\begin{corollary}\label{cor:reduction_number}
Let $I$ be an integrally closed monomial ideal in $A$ with $\height_AI = 3$ and $\mu_A(I) \leq 7$. 
Then Then $I$ admits a minimal reduction $Q$ generated by three elements such that $I^3=QI^2$.
\end{corollary}

\section*{Acknowledgments}

The authors are deeply grateful to the late Professor Shiro Goto for insights that
shaped this work.
Years ago, when the second author was exploring concrete examples of
integrally closed monomial ideals, Professor Goto suggested approaching the problem
from a valuation-theoretic point of view, which became essential to the development
presented in this paper.
This work is dedicated to his memory.

The authors also thank Naoki Endo for helpful comments on an earlier version of
this manuscript.

\end{document}